\documentclass[10pt,fleqn]{article}

\usepackage{mathptmx}
\usepackage[T1]{fontenc}
\usepackage{eucal}

\usepackage{latexsym,epsfig,verbatim}
\usepackage{amsmath,amsthm,amssymb}
\usepackage{colonequals}

\usepackage{url}
\usepackage{color}
\usepackage[colorlinks,citecolor=blue]{hyperref}

\usepackage{accents}

\usepackage{tikz}
\usetikzlibrary{matrix,arrows,decorations.pathmorphing}

\theoremstyle{plain}
\newtheorem{theorem}{Theorem}[section]

\newtheorem*{MainGeneralized}{Theorem \ref{T:generalized}}
\newtheorem*{ScottSwarupGeneralized}{Theorem \ref{T:General.Scott.Swarup}}
\newtheorem{lemma}[theorem]{Lemma}
\newtheorem{corollary}[theorem]{Corollary}

\newtheorem{proposition}[theorem]{Proposition}

\newtheorem*{claim}{Claim}

\theoremstyle{definition}
\newtheorem*{remark}{Remark}
\newtheorem{question}{Question}

\newcommand{\define}[1]{\textbf{#1}}

\newcommand{\T}{\mathrm{Teich}}
\newcommand{\Mod}{\mathrm{Mod}}

\newcommand{\PSL}{\mathrm{PSL}}

\newcommand{\Hull}{\mathrm{Hull}}

\newcommand{\Aut}{\mathrm{Aut}}
\newcommand{\Out}{\mathrm{Out}}

\newcommand{\HH}{\mathbb H}
\newcommand{\RR}{\mathbb R}
\newcommand{\ZZ}{\mathbb Z}

\newcommand{\C}{\mathcal C}

\newcommand{\diam}{\mathrm{diam}}
\newcommand{\Isom}{\mathrm{Isom}}

\newcommand{\co}{\colon\thinspace}

\begin{document}

\title{\textbf{\Large Pseudo-Anosov subgroups of fibered 3--manifold groups}}
\author{Spencer Dowdall, Richard P. Kent IV, and Christopher J. Leininger\thanks{The first author was supported by NSF MSPRF No. 1204814, the  second by NSF grant DMS-1104871,  and the third  by NSF grant DMS-0905748.
All three authors were partially supported by the GEAR network.
}}
\date{April 5, 2013}

\maketitle

\begin{abstract}
Let $S$ be a hyperbolic surface and let $\mathring{S}$ be the surface obtained  from $S$ by removing a point.
The mapping class groups $\Mod(S)$ and  $\Mod(\mathring{S})$ fit into a short exact sequence
	$
		1 \to \pi_1(S) \to \Mod(\mathring{S}) \to \Mod(S) \to 1.
	$
If $M$ is a hyperbolic $3$--manifold that fibers over the circle with fiber $S$, then its fundamental group fits into a short exact sequence
	$
		1 \to \pi_1(S) \to \pi_1(M) \to \ZZ \to 1
	$
that injects into the one above.

We show that, when viewed as subgroups of $\Mod(\mathring{S})$, finitely generated purely pseudo-Anosov subgroups of $\pi_1(M)$ are convex cocompact in the sense of Farb and Mosher.   
More generally, if we have a $\delta$--hyperbolic surface group extension 
	$
		1 \to \pi_1(S) \to \Gamma_\Theta \to \Theta \to 1,
	$
any quasiisometrically embedded purely pseudo-Anosov subgroup of $\Gamma_\Theta$ is convex cocompact in $\Mod(\mathring{S})$.
We also obtain a generalization of a theorem of Scott and Swarup by showing that finitely generated subgroups of $\pi_1(S)$ are quasiisometrically embedded in hyperbolic extensions $\Gamma_\Theta$.
\end{abstract}

%%%%%%%%%%%%%%%%
\section{Introduction}
%%%%%%%%%%%%%%%%

In \cite{FMcc}, Farb and Mosher defined a notion of convex cocompactness for subgroups $\Theta < \Mod(S)$ of the mapping class group of a closed hyperbolic surface $S$ by analogy with convex cocompactness in the theory of Kleinian groups.   
This analogy was extended by the second and third authors \cite{kentleiningershadows,kentleiningeruniform}.   
Combining the results of Farb--Mosher \cite{FMcc} and Hamenst\"adt \cite{hamenstadt}, it follows that  the associated $\pi_1(S)$--extension $\Gamma_\Theta$ of $\Theta < \Mod(S)$ given by
	\[
		\begin{tikzpicture}[>= to, line width = .075em, 
			baseline=(current bounding box.center)]
		\matrix (m) [matrix of math nodes, column sep=1.5em, row sep = 1em, 		text height=1.5ex, text depth=0.25ex]
		{
			1 & \pi_1(S)  & \Gamma_\Theta & \Theta & 1 \\
		};
		\path[->,font=\scriptsize]
		(m-1-1) edge (m-1-2)
		(m-1-2) edge (m-1-3)
		(m-1-3) edge (m-1-4)
		(m-1-4) edge (m-1-5)
		;
		\end{tikzpicture}
	\]
is $\delta$--hyperbolic for some $\delta$ if and only if $\Theta$ is convex cocompact.  
For punctured surfaces, one has a similar statement for associated orbifold extensions \cite{FMcc}, or one can replace hyperbolicity with relative hyperbolicity \cite{Mj.Sardar.2012}; see Section \ref{sec:convexCocompactness}.

If $\Theta < \Mod(S)$ is convex cocompact, then it must be finitely generated and \define{purely pseudo-Anosov}, meaning that every infinite order element is pseudo-Anosov.  
Conversely, if $\Theta$ is purely pseudo-Anosov then the (orbifold) extension $\Gamma_\Theta$ has no Baumslag--Solitar subgroups.  
As such subgroups are the natural obstructions to being hyperbolic, Farb and Mosher \cite{FMcc} asked
\begin{question} \label{Q:FMgromov}
If $\Theta < \Mod(S)$ is a purely pseudo-Anosov, finitely generated free group, is $\Theta$ convex cocompact?
\end{question}
The hypotheses imply that $\Gamma_\Theta$ has a $3$--dimensional $K(\Gamma_\Theta,1)$, and so this is a special case of a question of Gromov, see \cite{kentleiningerabc}.  
More generally one can ask if every finitely generated purely pseudo-Anosov subgroup is convex cocompact.  
These questions seem difficult in general as the class of purely pseudo-Anosov subgroups is somewhat mysterious.  

We attack Question \ref{Q:FMgromov} for certain classes of purely pseudo-Anosov subgroups related to the Kleinian origins of convex cocompactness. 
Recall that if $M = \HH^3/\Gamma$ is an orientable finite volume hyperbolic $3$--manifold that fibers over $S^1$ with fiber $S$, and $\mathring{S}$ denotes the surface equipped with a distinguished basepoint, then there is a natural injection $\Gamma \to \Mod(\mathring{S})$; see sections \ref{sec:hyperbolicGeometry} and \ref{sec:groupExtensions}. 
We may then view any subgroup $G < \Gamma$ as a subgroup of $\Mod(\mathring{S})$.
\begin{theorem} \label{T:main}
Let $\Gamma$ be the fundamental group of hyperbolic $3$--manifold that fibers over the circle with fiber $S$, considered as a subgroup of $\Mod(\mathring{S})$.
If $G$ is a finitely generated purely pseudo-Anosov subgroup of $\Gamma$, then it is convex cocompact.
\end{theorem}

This is a generalization of the second and third authors' work with Schleimer \cite{kentleiningerschleimer}, where the subgroup $G$ was contained in the fiber group $\pi_1(S) < \pi_1(M) \cong \Gamma$.  
In that case, the group $G$ could be naturally identified with a Fuchsian group, and $2$--dimensional hyperbolic geometry could be used to provide the additional leverage needed to prove convex cocompactness.
In Theorem \ref{T:main}, $G$ is naturally a Kleinian group, and $3$--dimensional geometric techniques can be applied in a similar way to prove convex cocompactness, though there are a number of technical obstacles in the generalization.

The ideas used to deal with these obstacles apply in a more general setting.  
Specifically, given $\Theta < \Mod(S)$, the extension $\Gamma_\Theta$ also naturally injects into $\Mod(\mathring{S})$.
\begin{theorem}\label{T:generalized}  Suppose $S$ is a closed surface, $\Theta < \Mod(S)$ is a convex cocompact subgroup and $G< \Gamma_\Theta$ is a finitely generated quasiisometrically embedded subgroup.  
If $G$ is purely pseudo-Anosov as a subgroup of $\Mod(\mathring{S})$, then it is convex cocompact.
\end{theorem}

In \cite{Scott.Swarup.1990}, Scott and Swarup prove that if $\Gamma$ is the fundamental group of a hyperbolic $3$--manifold fibered over the circle with fiber $S$, then any finitely generated infinite--index  subgroup of the fiber subgroup $\pi_1(S) < \Gamma$ is convex cocompact.
A consequence of our work is the following generalization of this to arbitrary hyperbolic extensions $\Gamma_\Theta$. 
See \cite{Mitra.1999} for a discussion of an analog for hyperbolic free--by--cyclic groups.

\begin{theorem}\label{T:General.Scott.Swarup}
Let 
	\[
		\begin{tikzpicture}[>= to, line width = .075em, 
			baseline=(current bounding box.center)]
		\matrix (m) [matrix of math nodes, column sep=1.5em, row sep = 1em, 		text height=1.5ex, text depth=0.25ex]
		{
			1 & \pi_1(S)  & \Gamma_\Theta & \Theta & 1 \\
		};
		\path[->,font=\scriptsize]
		(m-1-1) edge (m-1-2)
		(m-1-2) edge (m-1-3)
		(m-1-3) edge (m-1-4)
		(m-1-4) edge (m-1-5)
		;
		\end{tikzpicture}
	\]
be a $\delta$--hyperbolic surface group extension.
If $H$ is a finitely generated infinite--index subgroup of $\pi_1(S)$, then $H$ is quasiisometrically embedded in $\Gamma_\Theta$.
\end{theorem}

This theorem follows from Proposition \ref{P:qi subsurfaces general}, and is proven in the final section.

\subsection{Outline of the proofs}

As discussed above, both Theorems \ref{T:main} and \ref{T:generalized} follow the approach used in  \cite{kentleiningerschleimer}.  
We briefly describe the main ideas and technical results needed to carry out the proofs.  See Section \ref{S:background} for definitions.

Suppose $G < \Gamma$ is as in Theorem \ref{T:main}.  
According to \cite{kentleiningershadows} or \cite{hamenstadt}, convex cocompactness is equivalent to the orbit map $G \to G \cdot u$ to the curve complex $\C(\mathring{S})$ being a quasiisometric embedding for some vertex $u \in \C^{(0)}(\mathring{S})$.  
The difficulty establishing this criterion is proving that distances in $\C(\mathring{S})$ are coarsely bounded below by those in $G$ equipped with a word metric.  

The first step is to prove that $G$ is convex cocompact as a Kleinian group (see Lemma \ref{L:cccKleinian}).  
The corresponding step in \cite{kentleiningerschleimer} is straightforward, but here we need to appeal to some fairly technical results in Kleinian groups (Tameness \cite{agoltame,calegarigabai} and the Covering Theorem \cite{canary,canarysurvey}).   It follows that distances in $G$ are comparable to distances between orbit points in $\HH^3$, and, more importantly, in the convex hull $\Hull(G) \subset \HH^3$ of the limit set for $G$.  

It then suffices to bound the distance between orbit points in $\HH^3$  by the distance of the corresponding orbit points in $\C(\mathring{S})$.  
We do this as follows. 
For any edge--path between orbit points in $\C(\mathring{S})$ of some length $n$, we construct a piecewise geodesic path in $\HH^3$ between corresponding orbit points in $\Hull(G)$, built from $n+1$ geodesic segments.  Each of the geodesic segments is contained in a convex set $\Hull(u) \subset \HH^3$ canonically associated to one of the vertices $u$ of the path in $\C(\mathring{S})$.  Specifically, $\Hull(u) = \Hull(\pi_1(S)_u)$, where $\pi_1(S)_u < \pi_1(S)$ is the stabilizer of $u$ under the action of $\pi_1(S)$ on $\C(\mathring{S})$ coming from the natural injection $\pi_1(S) < \Gamma \to \Mod(\mathring{S})$.  

The remaining step is to prove that, when this path is projected back to $\Hull(G)$, each of the $n+1$ geodesic segments projects to a path of uniformly bounded length.
The key to this is Proposition \ref{P:bounded diameter 1}, which states that for any simplex $u \subset \C(\mathring{S})$, the set $\Hull(u) \cap \Hull(G)$ has diameter bounded independently of $u$.  

This in turn relies on Proposition \ref{P:qi subsurfaces}. To briefly describe this, first observe that the fibration $M \to S^1$ lifts to a fibration $\HH^3 \to \RR$ of the universal cover whose fibers are naturally viewed as hyperbolic planes.
Roughly speaking, Proposition \ref{P:qi subsurfaces} states that for any simplex $u \subset \C(\mathring{S})$, there is a $t_u \in \RR$ so that the convex hull of the limit set of $\pi_1(S)_u$ in the fiber over $t_u$ is uniformly close to $\Hull(u)$.
The proof of this uses recent work of the second and third authors \cite{kentleiningerwidth}.

Theorem \ref{T:generalized} follows the same basic approach.  The abstraction from hyperbolic $3$--space to a more general Gromov hyperbolic space is an obstacle that causes little difficulty.   
However, the technical results in Kleinian groups are unavailable here, and we must assume that $G$ is quasiconvex to begin with.  
The proof then reduces to proving the analogue of Proposition \ref{P:qi subsurfaces} in this setting, which is Proposition \ref{P:qi subsurfaces general}.\\

\noindent
{\bf Acknowledgements.} The authors would like to thank Ian Agol for helpful conversations.  In particular, the proof of Lemma \ref{L:cccKleinian} was inspired by an idea of Ian Agol.  The authors would also like to thank the referee for several useful comments and suggestions.

%%%%%%%%%%%%%%%%%%%%%
\section{Background} \label{S:background}
%%%%%%%%%%%%%%%%%%%%%

%%%%%%%%%%%%%%%%%
\subsection{Hyperbolic geometry}
\label{sec:hyperbolicGeometry}
%%%%%%%%%%%%%%%%%

An orientable hyperbolic $n$--manifold is the quotient of hyperbolic $n$--space $\HH^n$ by a discrete torsion--free subgroup of $\Isom^+(\HH^n)$.  
We will be primarily interested in the case of $n = 2,3$, where such a group is called a (torsion--free) \define{Kleinian group}.   
A Kleinian group $G$ is called a \define{lattice} if the volume of $\HH^n/G$ is finite.  

Hyperbolic space is compactified by adding a \define{sphere at infinity} $\partial_\infty\HH^n$ to obtain a ball $\overline{\HH}^n = \HH^n \cup \partial_\infty\HH^n$.  The limit set of a Kleinian group $G$ is the set of accumulation points of any orbit
	\[ 
		\Lambda_G = \overline{G \cdot x} \setminus G \cdot x \subset \partial_\infty\HH^n.
	\]
The limit set is independent of the choice of point $x \in \HH^n$ used to define it.
The \define{convex hull of the limit set} is the smallest closed convex set in $\HH^n$ whose closure in $\overline{\HH}^n$ contains $\Lambda_G$, and will be denoted
	\[ 
		\Hull(G) \subset \HH^n.
	\]
Since $\Lambda_G$ is $G$--invariant, so is $\Hull(G)$.

We say that a Kleinian group $G$ is \define{convex cocompact} if $\Hull(G)/G$ is compact.  If $G$ is convex cocompact, then it is \define{purely hyperbolic}, meaning that every infinite order element is hyperbolic.  
In dimension $2$, the converse is true for finitely generated $G$.

\begin{theorem}[see  {\cite[Theorem 10.1.2]{beardon}}] \label{T:Fuchsianccc}
A torsion--free Kleinian group $G < \Isom^+(\HH^2)$ (usually called a torsion--free \define{Fuchsian group}) is convex cocompact if and only if it is finitely generated and contains no parabolics.
\qed
\end{theorem}

The situation is more complicated when $G < \Isom^+(\HH^3)$.  For this case, we consider a special situation.  
Suppose $\Gamma < \Isom^+(\HH^3)$ is a torsion--free lattice.  
Let $M = \HH^3/\Gamma$.  
Suppose that $M$ fibers over the circle with fiber a surface $S$
	\[
		\begin{tikzpicture}[>= to, line width = .075em, 
			baseline=(current bounding box.center)]
		\matrix (m) [matrix of math nodes, column sep=1.5em, row sep = 1em, 		text height=1.5ex, text depth=0.25ex]
		{
			S  & M & S^1. \\
		};
		\path[->,font=\scriptsize]
		(m-1-1) edge (m-1-2)
		(m-1-2) edge (m-1-3)
		;
		\end{tikzpicture}
	\]
Then the fundamental groups fit into a short exact sequence
	\[
		\begin{tikzpicture}[>= to, line width = .075em, 
			baseline=(current bounding box.center)]
		\matrix (m) [matrix of math nodes, column sep=1.5em, row sep = 1em, 		text height=1.5ex, text depth=0.25ex]
		{
			1 & \pi_1(S)  & \Gamma & \ZZ & 1, \\
		};
		\path[->,font=\scriptsize]
		(m-1-1) edge (m-1-2)
		(m-1-2) edge (m-1-3)
		(m-1-3) edge (m-1-4)
		(m-1-4) edge (m-1-5)
		;
		\end{tikzpicture}
	\]
where $\Gamma = \pi_1(M)$ and $\ZZ = \pi_1(S^1)$.
In particular, the subgroup $\pi_1(S) <  \Gamma$ is a finitely generated, infinite--index, \textit{normal} subgroup of $\Gamma$ that we call a \define{fiber group}.  
By normality, we have equality of the limit sets $\Lambda_{\pi_1(S)} = \Lambda_\Gamma = S^2_\infty$ and there is an infinite--sheeted covering
	\[ 
		\Hull(\pi_1(S))/\pi_1(S) = \Hull(\Gamma)/\pi_1(S) \to \Hull(\Gamma)/\Gamma.
	\]
By the Tameness Theorem \cite{agoltame,calegarigabai} and the Covering Theorem \cite{canary,canarysurvey}, this is essentially the only type of subgroup of $\Gamma$ which fails to be convex cocompact:

\begin{theorem} \label{T:tame+covering}
Suppose $\Gamma < \Isom^+(\HH^3)$ is a torsion--free lattice, and $G < \Gamma$ is a finitely generated subgroup without parabolics.  
Then either $G$ is convex cocompact, or else there is a subgroup $\widetilde G < G$ with index at most $2$ and a finite index subgroup $\widetilde \Gamma < \Gamma$ such that $\widetilde G < \widetilde \Gamma$ is a fiber subgroup. \qed
\end{theorem}

%%%%%%%%%%%%%%%%
\subsection{Coarse geometry}
%%%%%%%%%%%%%%%%

Let $\delta \geq 0$.
A geodesic triangle $\triangle$ in a geodesic metric space $X$ is \textbf{$\delta$--thin} if each of its sides lies in the $\delta$--neighborhood of the union of the other two sides.
A geodesic metric space $X$ is \textbf{$\delta$--hyperbolic} if every geodesic triangle is $\delta$--thin.

Let $K$ and $C$ be positive numbers.
A map $f\co X \to Y$ between metric spaces is a \textbf{$(K,C)$--quasiisometric embedding} if
	\[
		\frac{1}{K} d_X(a,b) - C \leq d_Y(f(a),f(b)) \leq K d_X(a,b) + C
	\]
for all $a$ and $b$ in $X$.
A quasiisometric embedding $f$ is a \textbf{quasiisometry} if its image is $D$--dense for some $D$. A quasiisometric embedding $\mathcal G\co I\to X$ from an interval $I\subseteq \RR$ into $X$ is called a \define{quasigeodesic}. 

If $A > 0$, a subset $Y$ of a geodesic metric space $X$ is \textbf{$A$--quasiconvex} if each geodesic joining points of $Y$ lies in the $A$--neighborhood of $Y$.

The \define{Gromov boundary} $\partial_\infty X$ of a proper $\delta$--hyperbolic space $X$ is defined to be the set of equivalence classes of quasigeodesic rays $\mathcal G\co [0,\infty)\to X$, where two rays are equivalent if they have finite Hausdorff distance. 
In this way each biinfinite quasigeodesic determines two distinct endpoints at infinity. 
The following consequence of $\delta$--hyperbolicity is well known; for a proof see \cite[Theorem III.1.7]{bridsonhaefliger}.

\begin{theorem}[Stability of quasigeodesics]\label{T:QIstability}
Given $K,C,\delta>0$, there exists a \define{stability constant} $R>0$ with the following property: For any $(K,C)$--quasigeodesic $\mathcal G'$ in a $\delta$--hyperbolic space $X$, every geodesic $\mathcal G$ in $X$ with the same endpoints (possibly at infinity) has Hausdorff distance at most $R$ from $\mathcal G'$. \qed
\end{theorem}

A finitely generated group is made into a metric space by equipping its Cayley graph with the path metric induced by declaring that edges have length one and giving the group the subspace metric.  
Up to quasiisometry, this metric does not depend on the finite generating set. 
The group is said to be \define{$\delta$--hyperbolic} if there is a choice of finite generating set such that its Cayley graph is $\delta$--hyperbolic. 
We will make frequent use of the following well--known fact; see \cite[Proposition I.8.19]{bridsonhaefliger} for a proof.

\begin{theorem}[\v{S}varc--Milnor lemma]
\label{T:svarc-milnor}
If $X$ is a proper geodesic metric space, and $G$ acts properly and cocompactly by isometries on $X$, then $X$ and $G$ are quasiisometric. 
In fact, for any $x$ in $X$, the orbit map $G \to X$ given by $g \mapsto gx$ is a quasiisometry. \qed
\end{theorem}

The following is a straightforward consequence of the stability of quasigeodesics.
\begin{proposition} If $\Gamma$ is $\delta$--hyperbolic and $G< \Gamma$ is finitely generated, then $G$ is quasiisometrically embedded if and only if it is quasiconvex. \qed
\end{proposition}

We conclude this section by noting that, for Kleinian groups, convex cocompactness may be reformulated in terms of coarse geometry as follows.

\begin{theorem} \label{T:Kleiniancccqi}
A Kleinian group $G < \Isom^+(\HH^n)$ is convex cocompact if and only if the orbit map $G \to G \cdot x \subset \HH^n$ is a quasiisometric embedding. \qed
\end{theorem}

%%%%%%%%%%%%%%%%%
\subsection{Mapping class groups and the complex of curves}
%%%%%%%%%%%%%%%%%

Let $S$ be a finite--volume hyperbolic surface, we may then identify $\pi_1(S)$ with a lattice in $\Isom(\HH^2)$ and write $S = \HH^2/\pi_1(S)$.
The \define{complexity} of $S$ is $\xi(S) = 3g-3 + n$, where $g$ is the genus of $S$ and $n$ is the number of its punctures.   
We assume throughout that $\xi(S) \geq 1$, which means that $S$ has negative Euler characteristic and is not a thrice--punctured sphere.  
The \define{mapping class group} $\Mod(S)$ of $S$ is the group of isotopy classes of orientation--preserving homeomorphisms of $S$.  

The mapping class group acts on a number of spaces, but for our purposes, the most important one is the \define{complex of curves} $\C(S)$.  
This is a simplicial complex whose vertices are isotopy classes of essential simple closed curves---these are precisely the isotopy classes with simple closed geodesic representatives.   
When $\xi(S) > 1$, we say that $k+1$ distinct isotopy classes span a $k$--simplex if and only if they can be realized disjointly on the surface (equivalently, their geodesic representatives are all disjoint).  
When $\xi(S) = 1$, the surface $S$ is either a once--punctured torus or a four--times punctured sphere.  In these cases, $k+1$ isotopy classes are the vertices of a $k$--simplex if and only if they pairwise intersect once or twice, respectively.  

We view $\C(S)$ as either a combinatorial object or a geometric object.  
For the latter, we declare each simplex to be isometric to a regular Euclidean simplex, and give $\C(S)$ the induced path metric.  
We make extensive use of the following celebrated theorem of Masur and Minsky \cite{MM1}.

\begin{theorem}[Masur--Minsky \cite{MM1}] \label{T:MMhyp}
For any $S$ there is a $\delta > 0$ such that $\C(S)$ is $\delta$--hyperbolic.
\qed
\end{theorem}

An element of $\Mod(S)$ is pseudo-Anosov if it has positive (asymptotic) translation length on $\C(S)$.  
That is, the pseudo-Anosov elements are precisely the analogues of the hyperbolic isometries of $\HH^n$.

%%%%%%%%%%%%%
\subsection{Exact sequences}
%%%%%%%%%%%%%%%%%

We will also be concerned with the \define{marked surface} $\mathring{S}$, which is simply the surface $S$ equipped with a distinguished basepoint (or an additional preferred puncture).
 In the corresponding based mapping class group $\Mod(\mathring{S})$, homeomorphisms and isotopies are required to fix the basepoint. 
 There is a natural surjection $\Mod(\mathring{S}) \to \Mod(S)$ obtained by simply `forgetting' the basepoint. Birman \cite{birmanExactSeq,birmanBook} showed that the kernel of this map may be identified with $\pi_1(S)$ thus giving an exact sequence
	\begin{equation}\label{eqn:BES}
		\begin{tikzpicture}[>= to, line width = .075em, 
			baseline=(current bounding box.center)]
		\matrix (m) [matrix of math nodes, column sep=1.5em, row sep = 1em, 		text height=1.5ex, text depth=0.25ex]
		{
			1 & \pi_1(S)  & \Mod(\mathring{S}) & \Mod(S) & 1. \\
		};
		\path[->,font=\scriptsize]
		(m-1-1) edge (m-1-2)
		(m-1-2) edge (m-1-3)
		(m-1-3) edge (m-1-4)
		(m-1-4) edge (m-1-5)
		;
		\end{tikzpicture}
	\end{equation}
The injective homomorphism $\Mod(S) \to \Out(\pi_1(S))$ naturally gives rise to an inclusion of short exact sequences
	\begin{equation}\label{E:innautout}
		\begin{tikzpicture}[>= to, line width = .075em, 
			baseline=(current bounding box.center)]
		\matrix (m) [matrix of math nodes, column sep=1.5em, row sep = 1.5em, 		text height=1.5ex, text depth=0.25ex]
		{
			1 & \pi_1(S)  & \Mod(\mathring{S}) & \Mod(S) & 1 \\
			1 & \pi_1(S)  & \Aut(\pi_1(S)) & \Out(\pi_1(S)) & 1 \\
		};
		\path[->,font=\scriptsize]
		(m-1-1) edge 					(m-1-2)
		(m-1-2) edge 					(m-1-3)
		(m-1-3) edge 					(m-1-4)
		(m-1-4) edge 					(m-1-5)
		
		(m-2-1) edge 					(m-2-2)
		(m-2-2) edge 					(m-2-3)
		(m-2-3) edge 					(m-2-4)
		(m-2-4) edge 					(m-2-5)
		
		(m-1-3) edge					(m-2-3)
		(m-1-4) edge					(m-2-4)
		;
		\draw[double distance = .15em,font=\scriptsize]
		(m-2-2) -- 					(m-1-2)
		;
		\end{tikzpicture}
	\end{equation}

The following result of Kra \cite{kra} provides many pseudo-Anosov elements.  
Recall that a loop $\gamma \in \pi_1(S)$ is \define{filling} if the geodesic representative of the free homotopy class cuts $S$ into disks and once--punctured disks.
\begin{theorem}[Kra \cite{kra}]\label{T:kra}
An element $\gamma \in \pi_1(S)$ is pseudo-Anosov as an element of $\Mod(\mathring{S})$ if and only if $\gamma$ is filling as an element of $\pi_1(S)$.
\qed
\end{theorem}

%%%%%%%%%
\subsection{Surface and orbifold group extensions}
%%%%%%%%%%%%%%%%%

\label{sec:groupExtensions}
Given $\Theta < \Mod(S)$, the exact sequence \eqref{eqn:BES} can be used to describe a $\pi_1(S)$--extension $\Gamma_\Theta$.  
Specifically, we can define $\Gamma_\Theta$ as the preimage of $\Theta$ in $\Mod(\mathring{S})$, which gives rise to an inclusion of short exact sequences
	\begin{equation}\label{E:ses inclusions}
		\begin{tikzpicture}[>= to, line width = .075em, 
			baseline=(current bounding box.center)]
		\matrix (m) [matrix of math nodes, column sep=1.5em, row sep = 1.5em, 		text height=1.5ex, text depth=0.25ex]
		{
			1 & \pi_1(S)  & \Gamma_\Theta & \Theta & 1 \\
			1 & \pi_1(S)  & \Mod(\mathring{S}) & \Mod(S) & 1 \\
		};
		\path[->,font=\scriptsize]
		(m-1-1) edge 					(m-1-2)
		(m-1-2) edge 					(m-1-3)
		(m-1-3) edge 					(m-1-4)
		(m-1-4) edge 					(m-1-5)
		
		(m-2-1) edge 					(m-2-2)
		(m-2-2) edge 					(m-2-3)
		(m-2-3) edge 					(m-2-4)
		(m-2-4) edge 					(m-2-5)
		
		(m-1-3) edge					(m-2-3)
		(m-1-4) edge					(m-2-4)
		;
		\draw[double distance = .15em,font=\scriptsize]
		(m-2-2) -- 					(m-1-2)
		;
		\end{tikzpicture}
	\end{equation}
When $\Theta = \langle \varphi \rangle$, and $\varphi$ is pseudo-Anosov, then $\Gamma_\Theta = \pi_1(M)$, where 
	\[
		M_\varphi = S \times [0,1]/(x,1) \sim (\varphi(x),0)
	\]
is the mapping torus of $\varphi$. 
By Thurston's Geometrization Theorem \cite{morgansurvey,otalhaken,otalfibered,kapovichbook}, we have $M = \HH^3/\Gamma$ with $\Gamma_\Theta \cong \pi_1(M) \cong \Gamma < \Isom^+(\HH^3)$.
In particular, this allows us to view $\Gamma$ and its subgroups as both Kleinian groups and subgroups of $\Mod(\mathring{S})$.

When $S$ has punctures, there are other extensions of $\Theta$.  
Namely, replacing each puncture of $S$ with a cone point of some order, we consider $S$ as a hyperbolic orbifold.  
There is an inclusion $\Mod(S) \to \Out(\pi_1(S)^{\mathrm{orb}})$, and we build an extension $\Gamma_\Theta^{\, \mathrm{orb}}$ as the preimage in $\Aut(\pi_1(S)^{\mathrm{orb}})$
		\[
		\begin{tikzpicture}[>= to, line width = .075em, 
			baseline=(current bounding box.center)]
		\matrix (m) [matrix of math nodes, column sep=1.5em, row sep = 1.5em, 		text height=1.5ex, text depth=0.25ex]
		{
			1 & \pi_1(S)^{\mathrm{orb}} & \Gamma_\Theta^{\, \mathrm{orb}} & \Theta& 1\\
			1 & \pi_1(S)^{\mathrm{orb}} & \Aut(\pi_1(S)^{\mathrm{orb}})& \Out(\pi_1(S)^{\mathrm{orb}}) & 1 \\
		};
		\path[->,font=\scriptsize]
		(m-1-1) edge 					(m-1-2)
		(m-1-2) edge 					(m-1-3)
		(m-1-3) edge 					(m-1-4)
		(m-1-4) edge 					(m-1-5)
		
		(m-2-1) edge 					(m-2-2)
		(m-2-2) edge 					(m-2-3)
		(m-2-3) edge 					(m-2-4)
		(m-2-4) edge 					(m-2-5)
		
		(m-1-3) edge					(m-2-3)
		(m-1-4) edge					(m-2-4)
		;
		\draw[double distance = .15em,font=\scriptsize]
		(m-2-2) -- 					(m-1-2)
		;
		\end{tikzpicture}
	\]

%%%%%%%%%%%%%%%%%%
\subsection{Convex cocompactness}
\label{sec:convexCocompactness}
%%%%%%%%%%%%%%%%%

Farb and Mosher defined convex cocompactness for $G < \Mod(S)$ in terms of the action of $\Mod(S)$  on Teichm\"uller space.  
An equivalent formulation mirroring Theorem \ref{T:Kleiniancccqi} is that $G$ is \define{convex cocompact} if and only if $G$ is finitely generated and the orbit map $G \to G \cdot v \subset \C(S)$ is a quasiisometric embedding (for any $v \in \C(S)$), see \cite{kentleiningershadows} or \cite{hamenstadt}.  
Analogous to the Kleinian group setting, if $G$ is convex cocompact, it is $\delta$--hyperbolic and purely pseudo-Anosov.

The following provides the link with the coarse geometry of surface group extensions and combines the results of Farb--Mosher \cite{FMcc} and Hamenst\"adt \cite{hamenstadt}.

\begin{theorem}[Farb--Hamenst\"adt--Mosher \cite{FMcc,hamenstadt}]\label{T:ccciffhyperbolic}
Suppose $S$ is a closed surface and $\Theta < \Mod(S)$ is a subgroup.  Then $\Theta$ is convex cocompact if and only if $\Gamma_\Theta$ is $\delta$--hyperbolic.
\qed
\end{theorem}

Although we will not need it, this theorem is also true when $S$ has punctures, provided we replace $\Gamma_\Theta$ with $\Gamma_\Theta^{\, \mathrm{orb}}$.  
More recently, Mj--Sardar \cite{Mj.Sardar.2012} proved that $\Theta$ is convex cocompact if and only if $\Gamma_\Theta$ is hyperbolic relative to the peripheral subgroups.

%%%%%%%%%%%%%%%%%%%%%%%%%%%%%%%%%%%
\section{Convex cocompactness as a Kleinian group.}
%%%%%%%%%%%%%%%%%%%%%%%%%%%%%%%%%%%

We now embark on the proof of Theorem \ref{T:main}, letting $\Gamma$ be the fundamental group of a hyperbolic $3$--manifold $M$ fibering over the circle with fiber $S$, considered as a subgroup of $\Mod(\mathring{S})$.
Let $G$ be a finitely generated purely pseudo-Anosov subgroup of $\Gamma$.

\begin{lemma} \label{L:cccKleinian}  $G$ is a convex cocompact Kleinian group.
\end{lemma}
\begin{proof}  By replacing $G$ with a subgroup of index at most two (which does not change the conclusion), Theorem \ref{T:tame+covering} implies that either $G$ is convex cocompact, or $G$ contains a parabolic, or $G$ is a fiber subgroup of a finite index subgroup $\widetilde \Gamma < \Gamma$.

We begin by observing that any element of $\Gamma$ which is pseudo-Anosov in $\Mod(\mathring{S})$ must be hyperbolic.  
For if not, it would be parabolic and hence contained in a maximal parabolic subgroup which is isomorphic to $\ZZ^2$.  
Since the centralizer of a pseudo-Anosov element is virtually cyclic (contains a cyclic subgroup of finite index), it cannot be contained in any subgroup isomorphic to $\ZZ^2$, and therefore cannot be parabolic in $\Gamma$.   

Therefore we must show that $G$ is not a fiber group of some $\widetilde \Gamma < \Gamma$.  
We assume that it is and derive a contradiction.

To this end, we let $\widetilde M \to M$ denote a finite cover that fibers over $S^1$ with a fiber subgroup $G$, and let $\widetilde \Gamma < \Gamma$ denote the corresponding subgroup of the fundamental group.  
Then $G$ is a normal subgroup of $\widetilde \Gamma$ with $\widetilde \Gamma/G \cong \ZZ$.  
Since $\widetilde \Gamma/(\widetilde \Gamma \cap \pi_1(S)) \cong \ZZ$, it follows that
	\[ 
		[\widetilde \Gamma,\widetilde \Gamma] < G \cap \pi_1(S) \cap \widetilde \Gamma
	\]
where $[\widetilde \Gamma,\widetilde \Gamma]$ is the commutator subgroup.

Now consider any strict essential subsurface $\Sigma \subset S$, and let $\Sigma_0$ be any component of the preimage of $\Sigma$ in the (finite--sheeted) covering $\widetilde S \to S$, corresponding to $\pi_1(S) \cap \widetilde \Gamma < \pi_1(S)$.  
Any nontrivial commutator in $\pi_1(\Sigma_0) < \pi_1(S)$ has infinite order (since all groups in question are torsion--free) and lies in $[\widetilde \Gamma,\widetilde \Gamma]$, hence also in $G$.  
On the other hand, every element of $\pi_1(\Sigma_0)$ projects to a loop in $\Sigma$, and hence a nonfilling loop in $S$.  
It follows that $G$ contains an infinite order element which is not pseudo-Anosov.  
This contradicts the assumption that $G$ is purely pseudo-Anosov and completes the proof.
\end{proof}

%%%%%%%%%%%%%%%%%
\section{Metrics and covers.}
%%%%%%%%%%%%%%%%%

The $3$--manifold $M$ is a quotient of $\HH^3$ by $\Gamma$, and so has a quotient hyperbolic metric we denote by $d$.  
We will want to consider an auxiliary metric $\hat d$ constructed as follows.  
The manifold $M$ is the mapping torus of $\varphi \co S \to S$, and we choose a suspension flow $\varphi_t$.  
That is, $\varphi_t$ is a flow transverse to the fibers such that $\varphi_t$ sends fibers to fibers for all $t$, and $\varphi_1$ is the first return map on each fiber.  
We choose a Riemannian metric so that the induced metric on each fiber is a hyperbolic metric, and so that $\varphi_t$ is a flow along flow lines that are orthogonal to the fibers. 
We let $\hat d$ denote the metric induced by this Riemannian metric.  
We also assume that the Riemannian metrics defining $\hat d$ and $d$ agree on some horoball neighborhoods of the cusps when $M$ is noncompact.   
This is possible since the fibration can be chosen so that in the hyperbolic metric $d$, the fibers intersect the cusps in totally geodesic surfaces, and so that the integral curves of the suspension flow are horocycles orthogonal to the fibers.

By compactness of the complement of the horoball cusp neighborhoods, the identity $(M,d) \to (M,\hat d)$ is bilipschitz.    
It follows that the same is true for any cover of $M$ if we pull back the metrics $d$ and $\hat d$.

We let $M_S \to M$ be the cover corresponding to $\pi_1(S)$, which fibers over $\RR$ by lifting the fibration $M \to S^1$.  
We record this, together with the homeomorphisms of the universal cover $\HH^3 \cong \HH^2 \times \RR$ and $M_S \cong S \times \RR$ in the following commutative diagram:
 	\[
		\begin{tikzpicture}[>= to, line width = .075em, 
			baseline=(current bounding box.center)]
		\matrix (m) [matrix of math nodes, column sep=1.5em, row sep = 1.5em, 		text height=1.5ex, text depth=0.25ex]
		{
			\HH^3  & M_S & M \\
			\HH^2 \times \RR & S \times \RR & \\
			\RR & \RR & S^1 \\
		};
		\path[->,font=\scriptsize]
		(m-1-1) edge (m-1-2)
		(m-1-2) edge (m-1-3)
		(m-1-3) edge (m-3-3)
		
		(m-2-1) edge (m-2-2)
		(m-2-1) edge (m-3-1)
		(m-2-2) edge (m-3-2)
		
		(m-3-2) edge (m-3-3)
		;
		\draw[double distance = .15em,font=\scriptsize]
		(m-1-1) -- 	(m-2-1)
		(m-1-2) --	(m-2-2)
		(m-3-1) --	(m-3-2)
		;
		\end{tikzpicture}
	\]
We pull the metric $\hat d$ on $M \to S^1$ back to $\HH^3$ and $M_S$.
We let $\HH^2_t = \HH^2 \times \{t\}$ and $S_t = S \times \{t\}$ denote the fibers of the fibrations $\HH^3 \to \RR$ and $M_S \to \RR$, respectively.
We equip these fibers with their path metrics induced by $\hat d$.

Let $\eta\colon \Gamma\to \ZZ$ denote the homomorphism induced by the fibration $M\to S^1$. 
We assume that in addition to being purely pseudo-Anosov and finitely generated,  $G < \Gamma$ is \textbf{not} contained in $\pi_1(S)$:  if $G < \pi_1(S)$ then we can apply \cite{kentleiningerschleimer}, and $G$ is convex cocompact in $\Mod(\mathring{S})$.  
Let $G_0 = G \cap \pi_1(S)$, which is the kernel of $\eta\vert_G\colon G\to \ZZ$.
%Let $G_0 = G \cap \pi_1(S)$, which is the kernel of the restricted homomorphism $\eta\vert_G\colon G\to \ZZ$.
	\[
		\begin{tikzpicture}[>= to, line width = .075em, 
			baseline=(current bounding box.center)]
		\matrix (m) [matrix of math nodes, column sep=1.5em, row sep = 1em, 		text height=1.5ex, text depth=0.25ex]
		{
			1 & G_0  & G & \ZZ & 1 \\
		};
		\path[->,font=\scriptsize]
		(m-1-1) edge 					(m-1-2)
		(m-1-2) edge 					(m-1-3)
		(m-1-3) edge 	node[auto]{$\eta\vert_G$}	(m-1-4)
		(m-1-4) edge 					(m-1-5)
		;
		\end{tikzpicture}
	\]
By replacing $\Gamma$ with a finite index subgroup (namely the preimage in $\Gamma$ of the image of $G$ in $\ZZ$), we can assume without loss of generality that $G \to \ZZ$ is surjective.  
We also let $g \in G$ be an element that maps to $1$ in $\ZZ$.  If $G_0$ is trivial, then $G$ is cyclic and we're done, so we assume $G_0$ is nontrivial.

We consider the covers $M_{G_0} \to M_G \to M$ corresponding to $G_0 < G < \Gamma$, as well as the cover $S_{G_0} \to S$ corresponding to $G_0 < \pi_1(S)$.
We can add this to the previous diagram to get
 	\[
		\begin{tikzpicture}[>= to, line width = .075em, 
			baseline=(current bounding box.center)]
		\matrix (m) [matrix of math nodes, column sep=1.5em, row sep = 1.5em, 		text height=1.5ex, text depth=0.25ex]
		{
			&  
			& M_G 	
			&	
			&	
				\\
			\HH^3  	
			& M_{G_0} 	
			& M_S 		
			& M 		
				\\
			\HH^2 \times \RR 	
			& S_{G_0} \times \RR 	
			& S \times \RR	
			&	
				\\
			\RR 		
			& \RR
			& \RR 	
			& S^1 
				\\
		};
		\path[->,font=\scriptsize]
		(m-1-3) edge	(m-2-4)
		
		(m-2-1) edge	(m-2-2)
		(m-2-2) edge	(m-1-3)
		(m-2-2) edge	(m-2-3)
		(m-2-3) edge	(m-2-4)
		(m-2-4) edge	(m-4-4)
		
		(m-3-1) edge	(m-3-2)
		(m-3-1) edge	(m-4-1)
		(m-3-2) edge	(m-3-3)
		(m-3-2) edge	(m-4-2)
		(m-3-3) edge	(m-4-3)
		;
		\draw[double distance = .15em,font=\scriptsize]
		(m-2-1) --		(m-3-1)
		(m-2-2) --		(m-3-2)
		(m-2-3) --		(m-3-3)
		
		(m-4-1) -- 		(m-4-2)
		(m-4-2) --		(m-4-3)
		(m-4-3) --		(m-4-4)
		;
		\end{tikzpicture}
	\]

Keeping with the same notation, we write $S_{G_0,t} = S_{G_0} \times \{t\}$ with its induced hyperbolic metric.

\bigskip

\noindent
\textbf{Notation.}  The metrics $d$ and $\hat d$ both pull back to metrics on all the covers, and we denote these by the same names $d$ and $\hat d$.  

For any of the spaces that fiber over $\RR$, we write $d_t$ for the path metric on the fiber over $t\in \RR$ induced by $\hat d$. 
By construction, $d_t$ is a hyperbolic metric on each fiber.
For metric--dependent constructions we will add a prefix to the name to signify what metric is being used in the construction.  
For example, we will refer to $d_t$--geodesics, $d_t$--diameter, $d_t$--Hausdorff distance, \textit{et cetera}.  
If there is no prefix this signifies that the metric $d$ is being used, though we will sometimes include the $d$ for clarification.

%%%%%%%%%%%%%%%%
\section{Simplex hulls}
%%%%%%%%%%%%%%%%

For any simplex $u \subset \C(\mathring{S})$, we consider the stabilizer of $u$ in $\pi_1(S) < \Mod(\mathring{S})$.  
This is the fundamental group of the subsurface determined by $u$, see \cite{kentleiningerschleimer} for a detailed discussion.  
Since the stabilizer $\pi_1(S)_u$ acts on $\HH^3$ as well as each $\HH^2_t$ for every $t$, we can consider the convex hull of its limit set in any one of these spaces.  
We write $\Hull(u) \subset \HH^3$ for the $d$--convex hull in $\HH^3$ and $\Hull_t(u) \subset \HH^2_t$ for the $d_t$--convex hull in $\HH^2_t$.
By the main theorem of \cite{Scott.Swarup.1990}, the limit set of $\Hull_{t}(u)$ in $\partial_\infty\HH^2_t$ maps homeomorphically onto the limit set of $\Hull(u)$ in $\partial_\infty\HH^3$.

Given any $\gamma \in \Gamma$ we have $\gamma(\Hull(u)) = \Hull(\gamma \cdot u)$, where $\gamma$ acts on $\C(\mathring{S})$ via the inclusion $\Gamma < \Mod(\mathring{S})$ from \eqref{E:ses inclusions}. Recall also that $\eta\colon \Gamma\to \ZZ$ is the homomorphism of fundamental groups induced by the fibration $M=\HH^3/\Gamma \to S^1$.

\begin{proposition} \label{P:qi subsurfaces}
There exist $K,C,R > 0$ with the following property.  
For any simplex $u \subset \C(\mathring{S})$ there exists a $t_u  \in \ZZ$ satisfying $t_{\gamma \cdot u} = t_u + \eta(\gamma)$ for all $\gamma \in \Gamma$, and such that the following holds.
\begin{enumerate}
\item The inclusion $\Hull_{t_u}(u) \to \HH^3$ is a $(K,C)$--quasiisometric embedding.
\item For any $z,w \in \Hull(u)$, there exists $z',w' \in \Hull_{t_u}(u)$ such that
	\[ 
		d(z,z'), d(w,w') \leq R
	\]
and the $d_{t_u}$--geodesic $[z',w']_{t_u} \subset \Hull_{t_u}(u)$ and the $d$--geodesic $[z,w] \subset \Hull(u)$ have $d$--Hausdorff distance at most $R$.
\end{enumerate}
\end{proposition}

One of the key ingredients in the proof of this proposition is the following result from \cite{kentleiningerwidth}, which shows that the convex hull $\Hull(u)\subset\HH^3=\HH^2\times \RR$ is not too wide.

\begin{proposition}[Kent--Leininger \cite{kentleiningerwidth}]\label{P:widthmain}
There exists $W > 0$ such that for any simplex $u \subset \C(\mathring{S})$, the image of $\Hull(u)$ in $\RR$ under the projection $\HH^3 = \HH^2 \times \RR \to \RR$ onto the second coordinate has diameter at most $W$. 
\qed
\end{proposition}

Given this Proposition, we now sketch the proof of Proposition  \ref{P:qi subsurfaces} before proceeding to the details.  

\begin{proof}[Sketch of Proposition \ref{P:qi subsurfaces}] We will choose $t_u$ uniformly close to the image of $\Hull(u)$ in $\RR$ under the projection $\HH^3 \to \RR$.  According to Proposition \ref{P:widthmain}, the image of $\Hull(u)$ in $\RR$ is contained in an interval $I_u$ centered on $t_u$ of uniformly bounded diameter.    Since $I_u$ has bounded diameter, it follows that the inclusion $\HH^2_{t_u} \to \HH^2 \times I_u$ is uniformly bilipschitz provided we give $\HH^2_{t_u}$ the hyperbolic metric $d_{t_u}$ and $\HH^2 \times I_u$ the induced path metric from $\hat d$ or from the hyperbolic metric $d$.  Since $\Hull_{t_u}(u) \subset \HH^2_{t_u}$ is convex, the same is true for the inclusion $\Hull_{t_u}(u) \to \HH^2 \times I_u$, and in particular, $\Hull_{t_u}(u)$ is uniformly quasi-convex in the Gromov hyperbolic space $\HH^2 \times I_u$.

Next we observe that $\Hull(u) \subset \HH^2 \times I_u$.  
Since $\Hull(u)$ is convex, the path metric on $\Hull(u)$ induced by $d$ is precisely the restriction of the hyperbolic metric $d$.  Since $\Hull(u)$ and $\Hull_{t_u}(u)$ have the same limit set, and both are quasi-convex, they are uniformly close to each other inside $\HH^2 \times I_u$.  From this and the fact that both spaces are uniformly quasi-isometrically embedded in $\HH^2 \times I_u$, we see that distances between points in $\Hull_{t_u}(u)$ are uniformly comparable to distances in $\Hull(u)$.  On the other hand, distances in $\Hull(u)$ are precisely distances in $\HH^3$, as required.
\end{proof}

With this sketch in mind, we proceed to the actual proof of Proposition \ref{P:qi subsurfaces}.  It turns out that comparing distances in $\Hull_{t_u}(u)$ and $\Hull(u)$ as just described is a bit messy.  The following Lemma allows us to restrict attention to the points which lie on biinfinite geodesics contained in the respective hulls.

\begin{lemma} \label{L:segments almost extend}
There exists $\delta > 0$ such that for $n = 2$ or $3$ the following holds.  
Let $Z$ be a closed subset of $\partial_\infty\HH^n$.
Then any geodesic segment $[z,w]$ in the convex hull $\Hull(Z)$ in $\HH^n$ has Hausdorff distance at most $\delta$ from a geodesic segment $[z_0,w_0]$ which is contained in a biinfinite geodesic in $\Hull(Z)$.
\end{lemma}
\begin{proof}
Extend $[z,w]$ as far as possible in both directions.  
If it extends indefinitely in both directions in $\Hull(Z)$, then $[z,w]$ is itself contained in a biinfinite geodesic in $\Hull(Z)$ and we are done.  
If not, then $[z,w]$ is contained in $\mathcal G$, a geodesic segment or ray that terminates in the boundary of the convex hull.   
By moving $\mathcal G$ a uniformly bounded amount to some $\mathcal G'$ if necessary, we can assume that each endpoint lies on a biinfinite geodesic in the boundary of $\Hull(Z)$, and that $\mathcal G'$ has length at least $10$, say.  
This follows from the fact that if $\Hull(Z)$ has dimension $2$, then the boundary of $\Hull(Z)$ is a union of biinfinite geodesics, and if $\Hull(Z)$ has dimension $3$, its boundary is a hyperbolic surface bent along a geodesic lamination \cite{epsteinmarden}.   
To each endpoint of $\mathcal G'$, append a ray of the biinfinite geodesic in the direction that makes the larger of the two angles with $\mathcal G'$ (which is at least $\pi/2$).  
The resulting broken geodesic is a uniformly bounded distance from a biinfinite geodesic, and this geodesic contains $[z,w]$ in a uniformly bounded neighborhood, as required.
\end{proof}

\begin{proof}[Proof of Proposition \ref{P:qi subsurfaces}] Let $W \geq 1/2$ be as in Proposition \ref{P:widthmain}.  
Therefore, for each $u$ there exists an integer $t \in \RR$ such that
	\begin{equation} \label{E:t option} 
		\Hull(u) \subset \HH^2 \times [t - W,t+W].
	\end{equation}

The action of $\Gamma$ on $\HH^2 \times \RR$ descends to an action on $\RR$ given by translation under $\eta$.  
Therefore, the projection of
	\[ 
		\Hull(\gamma \cdot u) = \gamma \Hull(u) 
	\]
to $\RR$ is the image of $\Hull(u)$ under the projection, after translating by $\eta(\gamma)$.  It follows that if $t$ is an integer satisfying \eqref{E:t option} for $u$, then
\begin{equation} \label{E:t choice} \Hull(\gamma \cdot u) \subset \HH^2 \times [t_u + \eta(\gamma) - W,t_u + \eta(\gamma) + W] .
\end{equation}

Let $T$ be a transversal for the action of $\Gamma$ on the set of simplices in $\C(\mathring{S})$, that is, a choice of simplex from each $\Gamma$ orbit.  
For any $u \in T$, pick $t_u = t$ satisfying \eqref{E:t option}, then define $t_{\gamma \cdot u} = t_u + \eta(\gamma)$ for any $u \in T$ and $\gamma \in \Gamma$.  
From \eqref{E:t choice}, it follows that
	\[ 
		\Hull(u) \subset \HH^2 \times [t_u -W,t_u + W] 
	\]
for every $u \in \C(\mathring{S})$.  
Since $\eta$ is a homomorphism, $t_{\gamma \cdot u} = t_u + \eta(\gamma)$ holds for every $\gamma \in \Gamma$ and $u \subset \C(\mathring{S})$.

Given any integer $t$, let $d^\times_t$ denote the path metric on $\HH^2 \times [t-W,t+W]$ induced by the hyperbolic metric $d$.
\begin{claim}  There exists $K > 0$ such that the inclusion
	\[
		\HH_t^2 \to \HH^2 \times [t-W,t+W] 
	\]
is a $K$--bilipschitz embedding (with respect to $d_t$ and $d^\times_t$).
\end{claim}
\begin{proof}[Proof of Claim]
Let $\hat d^\times_t$ denote the path metric on $\HH^2 \times [t-W,t+W]$ induced by $\hat d$.
Let $K_0$ be the bilipschitz constant for the identity map between $d^\times_t$ and $\hat d^\times_t$.
It follows that the inclusion
	\[ 
		\HH_t^2 \to \HH^2 \times [t-W,t+W] 
	\]
with respect to $d_t$ and $d^\times_t$ is $K_0$--Lipschitz.

There is a $K_1 > 1$ such that the suspension flow $\varphi_t$ is $K_1^t$--bilipschitz with respect to $\hat d$.  
This follows from compactness of the complement of the cusp neighborhoods and the fact that the flow is by isometries on the cusp neighborhoods.  
Lifting this flow to $\HH^3$, we can use it to define a projection
	\[ 
		\HH^2 \times [t-W,t+W] \to \HH_t^2
	\]
by projecting out the flow lines.  
This is $K_1^W$--Lipschitz with respect to $\hat d^\times_t$ and $d_t$, and hence $K_0K_1^W$--Lipschitz with respect to $d^\times_t$ and $d_t$.   
Setting $K = K_0K_1^W > K_0$, it follows that the inclusion
	\[ 
	 	\HH_t^2 \to \HH^2 \times [t-W,t+W] 
	 \]
is $K$--bilipschitz.
\end{proof}

Let $R' > 0$ be the stability constant for $(K,0)$-quasigeodesics in the Gromov hyperbolic metric space $(\HH^2 \times [t_u- W,t_u + W], d^\times_t)$, see Theorem \ref{T:QIstability}.  
Let $\delta > 0$ be the constant from Lemma \ref{L:segments almost extend} and set
	\[
		C = 4K(K \delta + R') .
	\]
The next claim will prove the first part of the proposition.

\begin{claim} For any $u \subset \C(\mathring{S})$ the inclusion
	\[ 
		\Hull_{t_u}(u) \to \HH^3 
	\]
is a $(K,C)$--quasiisometric embedding.
\end{claim}
\begin{proof}[Proof of Claim]
Since $\Hull_{t_u}(u) \subset \HH^2_{t_u}$ is isometrically embedded, it follows that the inclusion
	\[ 
		\Hull_{t_u}(u) \to  \HH^2 \times [t_u -W,t_u+W]
	\]
is a $K$--bilipschitz embedding.

Now let $z',w' \in \Hull_{t_u}(u)$ be any two points.  
According to Lemma \ref{L:segments almost extend} there are points $z'_0,w'_0 \in \Hull_{t_u}(u)$ such that
	\begin{equation} \label{E:close in hull t}
		d_{t_u}(z',z'_0), d_{t_u}(w',w'_0) \leq \delta
	\end{equation}
and such that the geodesic segment $[z'_0,w'_0]$ extends to a biinfinite $d_{t_u}$--geodesic $\mathcal G' \subset \Hull_{t_u}(u)$.

Since the limit set of $\Hull_{t_u}(u)$ in $\partial_\infty \HH_{t_u}^2$ embeds in $\partial_\infty\HH^3$, the path $\mathcal G'$ has two endpoints in $\partial_\infty \HH^3$. Let $\mathcal G$ be the $d$--geodesic with these endpoints, which is necessarily contained in $\Hull(u)\subset \HH^2\times[t_u-W,t_u+W]$. In particular, $\mathcal G$ is a $d_{t_u}^\times$--geodesic.  Since the $d_{t_u}$--geodesic $\mathcal G'$ is a $(K,0)$--quasigeodesic with respect to $d^\times_{t_u}$, Theorem~\ref{T:QIstability} implies that $\mathcal G$ and $\mathcal G'$ have $d_{t_u}^\times$--Hausdorff distance at most $R'$.

Let $z,w \in \mathcal G$ be points with
	\[ 
		d^\times_{t_u}(z_0',z), d^\times_{t_u}(w_0',w) \leq R'.
	\]
Since $d$ is less than $d^\times_{t_u}$, appealing to this and \eqref{E:close in hull t} we have
	\begin{equation} \label{E:z close in h3} 
		d(z',z)   \leq  d^\times_{t_u}(z',z) \leq  d^\times_{t_u}(z',z_0') + d^\times_{t_u}(z_0',z) \leq K \delta + R',
	\end{equation}
and likewise
	\begin{equation} \label{E:w close in h3} 
		d(w',w)  \leq  d^\times_{t_u}(w',w) \leq K \delta + R'.
	\end{equation}

Combining \eqref{E:z close in h3} and \eqref{E:w close in h3} with the triangle inequality, the fact that 
	\[
		\Hull_{t_u}(u) \to \HH^2 \times [t_u - W,t_u + W]
	\] 
is a $K$--bilipschitz embedding, and the fact that $d = d^\times_{t_u}$ on $\Hull(u)$ (since $\Hull(u)$ is convex) we find
	\begin{align*}
		d_{t_u} (z',w') 	& \leq  K d^\times_{t_u}(z',w') \\
					& \leq  K(d^\times_{t_u}(z,w) + d^\times_{t_u}(z',z) + d^\times_{t_u}(w,w')) \\
					& \leq  K d(z,w) + K(2(K\delta + R'))  \\
					& \leq   K(d(z',w') + d(z,z') + d(w',w)) + 2K(K \delta + R')\\
					& \leq K d(z',w') + 2K(K \delta + R') + 2K(K \delta + R') \\
					& =  K d(z',w') + C.
	\end{align*}
On the other hand, since $d \leq d^\times_{t_u}$ on $\HH^2 \times [t_u- W,t_u+W]$ it follows that
	\[ 
		d(z',w') \leq d^\times_{t_u}(z',w') \leq K d_{t_u}(z',w') + C.
	\]
Therefore, the inclusion $\Hull_{t_u}(u) \to \HH^3$ is a $(K,C)$--quasiisometric embedding, proving the claim.
\end{proof}

To finish the proof of the proposition, let $R = R' + \delta > 0$, where $R'$ is as in the proof of the claim.  
By increasing $R' > 0$ if necessary, we can assume that if $\mathcal G'$ is a $(K,0)$--quasigeodesic in $\HH^2 \times [t - W, t+ W]$ and $\mathcal G$ is the unique geodesic a $d^\times_t$--Hausdorff distance at most $R'$ away, then for any segment $[z,w] \subset \mathcal G$ there is a segment $[z',w'] \subset \mathcal G'$ for which the $d^\times_t$--Hausdorff distance between $[z,w]$ and $[z',w']$ is at most $R'$ and $d_{t}^\times(z,z'),d_{t}^\times(w,w') \leq R'$.

Let $u \subset \C(\mathring{S})$ be any simplex.  
Given $z,w \in \Hull(u)$, let $\mathcal G$ be a biinfinite geodesic in $\Hull(u)$ containing points $z_0,w_0 \subset \mathcal G$ such that
	\[ 
		d(z,z_0),d(w,w_0) \leq \delta,
	\]
as given by Lemma \ref{L:segments almost extend}.
%By convexity of $d$, the $d$--geodesics $[z,w]$ and $[z_0,w_0] \subset \mathcal G$ have $d$--Hausdorff distance at most $\delta$.

Let $\mathcal G'$ be the unique $d_{t_u}$--geodesic in $\Hull_{t_u}(u)$ with the same endpoints as $\mathcal G$.
Since $\mathcal{G}'$ is a $(K,0)$--quasigeodesic for $d_{t_u}^\times$, it has $d^\times_{t_u}$--Hausdorff distance at most $R'$ from $\mathcal G$.
Let $z',w' \in \mathcal G'$ be any points such that the $d^\times_{t_u}$--Hausdorff distance between $[z_0,w_0]$ and $[z',w']$ is at most $R'$ and
	\[ 
		d^\times_{t_u}(z_0,z'),d^\times_{t_u}(w_0,w') \leq R'.
	\]
Since $d^\times_{t_u} \geq d$, it follows from the triangle inequality that $d(z,z'),d(w,w') \leq R'$ and $[z,w]$ and $[z',w']$ have $d$--Hausdorff distance at most $R' + \delta = R$, completing the proof.
\end{proof}

%%%%%%%%%%%%%%%%
\section{Hull intersections}
%%%%%%%%%%%%%%%%

Let $C(M_{G_0}) = \Hull(G_0)/G_0$ and $C(M_G) = \Hull(G)/G$.  
Since $G_0 \triangleleft G$ is an infinite normal subgroup, the limit sets are equal so $\Hull(G_0) = \Hull(G)$, and hence there is an induced covering map $C(M_{G_0}) \to C(M_G)$---in fact it is a $\ZZ$--covering.  
The same is true for the $r$--neighborhoods, for any $r > 0$,
	\[ 
		N_r(C(M_{G_0})) \to N_r(C(M_G)).
	\]

We can compose the fibration $M \to S^1$ with the map $N_r(C(M_G)) \to M$ obtained by restricting the covering map $M_G \to M$.  
This produces a map $f \co N_r(C(M_G)) \to S^1$.  
This lifts to $\widetilde f \co  N_r(C(M_{G_0})) \to \RR$ which is simply the restriction to $N_r(C(M_{G_0}))$ of the projection onto the second coordinate of the product structure $M_{G_0} \cong S_{G_0} \times \RR$.
Let $S_{G_0,0} = S_{G_0} \times \{0\}$.
 	\[
		\begin{tikzpicture}[>= to, line width = .075em, 
			baseline=(current bounding box.center)]
		\matrix (m) [matrix of math nodes, column sep=1.5em, row sep = 1.5em, 		text height=1.5ex, text depth=0.25ex]
		{
			N_r(C(M_{G_0})) 
			& N_r(C(M_G))
			&
				\\
			M_{G_0}  
			& M_G 
			& M 
				\\
			\RR 
			& S^1  
			& S^1 
				\\
		};
		\path[->,font=\scriptsize]
		(m-1-1) edge (m-1-2)
		
		(m-2-1) edge (m-2-2)
		(m-2-1) edge (m-3-1)
		(m-2-2) edge (m-2-3)
		(m-2-2) edge (m-3-2)
		(m-2-3) edge (m-3-3)
		
		(m-3-1) edge (m-3-2)
		;
		\path[right hook->,font=\scriptsize]
		(m-1-1) edge (m-2-1)
		(m-1-2) edge (m-2-2)
		;
		\draw[double distance = .15em,font=\scriptsize]
		(m-3-2) --	(m-3-3)
		;
		\end{tikzpicture}
	\]

Because $\eta \co G \to \ZZ$ is surjective, the fibers of $\widetilde f$ project homeomorphically to the fibers of $f$ by the covering map $N_r(C(M_{G_0})) \to N_r(C(M_G))$.  
In particular, because $f$ has compact fibers, so does $\widetilde f$, and hence 
	\[ 
		\widetilde f^{-1}(0) = N_r(C(M_{G_0})) \cap S_{G_0,0} \subset S_{G_0,0}
	\]
is compact.  
We will assume, without loss of generality, that for whatever choice of $r$ we investigate, $0$ is a regular value for $\widetilde f$ restricted to the boundary of $N_r(C(M_{G_0}))$.  
Then $\widetilde f^{-1}(0) \subset S_{G_0,0}$ is a compact submanifold, hence has only finitely many components.

The next Proposition is the key ingredient needed to adapt the arguments from \cite{kentleiningerschleimer}.
\begin{proposition} \label{P:bounded diameter 1}
There exists $D > 0$ such that for any simplex $u \in \C(\mathring{S})$, the diameter of $\Hull(u) \cap N_1(\Hull(G))$ is at most $D$.
\end{proposition}

Before we launch into the proof we give a brief sketch.
\begin{proof}[Sketch of Proposition \ref{P:bounded diameter 1}]
Given two points of $\Hull(u) \cap N_1(\Hull(G))$, the geodesic between these points is contained in the intersection by convexity.  According to Proposition \ref{P:qi subsurfaces} any geodesic segment contained in $\Hull(u)$ is within a distance $R$ of a $d_{t_u}$--geodesic segment in $\Hull_{t_u}(u)$.  Thus, it suffices to find a uniform bound on the length of a $d_{t_u}$--geodesic segment in the intersection
\[ \Hull_{t_u}(u) \cap N_{R+1}(\Hull(G)) = \Hull_{t_u}(u) \cap N_{R+1}(\Hull(G_0)).\]
Furthermore, replacing $u$ by  its image under an element of $G$, we can assume $t_u = 0$.

Now suppose we have a $d_0$--geodesic segment in $\Hull_{t_u}(u) \cap N_{R+1}(\Hull(G_0))$.  Taking the quotient by $G_0$ we obtain a $d_0$--geodesic segment inside $\widetilde f^{-1}(0) \subset S_{G_0,0}$.   Since $\widetilde f^{-1}(0)$ is compact, the fundamental group is finitely generated, $G_1 < G_0$.  The compact subsurface $\widetilde f^{-1}(0)$ and the $d_0$--geodesic segment lift to the cover $S_{G_1,0} \to S_{G_0,0}$ corresponding to $G_1$.   Now we observe that $G_1 < \pi_1(S)$ is a finitely generated Fuchsian group, and is purely pseudo-Anosov as a subgroup of $\Mod(\mathring{S})$.  Thus, we can appeal directly to the arguments of \cite{kentleiningerschleimer} to bound the length of the $d_0$--geodesic segment.
\end{proof}

The proof of Proposition \ref{P:bounded diameter 1} requires the following result from \cite{kentleiningerschleimer}.

\begin{proposition}[Corollary 5.2 of \cite{kentleiningerschleimer}]
\label{P:bounded diameter fuchsian}
Let $G'< \pi_1(S)$ be a finitely generated subgroup which is purely pseudo-Anosov when considered as a subgroup of $\Mod(\mathring{S})$. Then for each $t\in \RR$, there exists $D'_t>0$ such that for any simplex $u \in \C(\mathring{S})$, the $d_t$--diameter of $\Hull_t(u) \cap N_1(\Hull_t(G'))$ is at most $D'_t$.\qed
\end{proposition}

We also need the following lemma.

\begin{lemma} \label{L:bounded diameter 2}
Let $R > 0$ be the constant from Proposition \ref{P:qi subsurfaces}.
There exists a $D' > 0$ with the following property.
Given any $u \in \C(\mathring{S})$, let $t_u \in \ZZ$ be the integer from Proposition \ref{P:qi subsurfaces}.  Then any $d_{t_u}$--geodesic segment
	\[ 
		[z',w'] \subset N_{R+1}(\Hull(G_0)) \cap \Hull_{t_u}(u)
	\]
has $d_{t_u}$--length at most $D'$.
\end{lemma}

\begin{proof}[Proof of Lemma \ref{L:bounded diameter 2}] Recall that we have chosen $g \in G$ with $\eta(g) = 1$.  
Given $u \in \C(\mathring{S})$ and any $d_{t_u}$--geodesic segment
	\[ 
		[z',w']  \subset N_{R+1}(\Hull(G_0)) \cap \Hull_{t_u}(u), 
	\]
Proposition \ref{P:qi subsurfaces} implies
	\[ 
		t_{g^{-t_u}\cdot u} = t_u + \eta(g^{-t_u}) = t_u - t_u = 0.
	\]
Combining this with the fact that $g^{-t_u}$ is an isometry on all of $\HH^3$, as well as from $\HH^2_{t_u}$ to $\HH^2_0$, and the fact that it preserves $\Hull(G_0) = \Hull(G)$, it follows that
	\begin{align*} 
		[g^{-t_u} z',g^{-t_u} w'] 	& =  g^{-t_u} ([z',w'])\\
							& \subset  g^{-t_u}\big(N_{R+1}(\Hull(G_0)) \cap \Hull_{t_u}(u)\big) \\
 							& =  N_{R+1}(\Hull(G_0)) \cap g^{-t_u}(\Hull_{t_u}(u))\\
 							& =  N_{R+1}(\Hull(G_0)) \cap \Hull_0(g^{-t_u} \cdot u)
	\end{align*}
Therefore, it suffices to prove the lemma for the case that $t_u = 0$.

By compactness, there are only finitely many components of
	\[ 
		\widetilde f^{-1}(0) = N_{R+1}(C(G_0)) \cap S_{G_0,0}.
	\]
To prove the lemma we must bound the length of a segment in $N_{R+1}(\Hull(G_0)) \cap \Hull_0(u)$.  
Since such a segment must project to one of the components $X_0 \subset \widetilde f^{-1}(0)$, it suffices to find a constant $D_0'  > 0$ such that the conclusion of the lemma is satisfied for segments that project to $X_0$.  
Taking the maximum of the constants over the finitely many components of $\widetilde f^{-1}(0)$ will complete the proof.

Let $p_0 \co \HH^2_0 \to S_{G_0,0}$ be the covering projection and $\widetilde X_0 \subset p_0^{-1}(X_0)$ a component of the preimage.  
Since $G_0$ acts transitively on the components of $p_0^{-1}(X_0)$, given a $d_0$--geodesic segment
	\[ 
		[z',w'] \subset p_0^{-1}(X_0) \cap \Hull_0(u)
	\]
there exists an element $g_0 \in G_0$ such that
	\[ 
		[g_0(z'),g_0(w')] = g_0([z',w']) \subset \widetilde X_0 \cap \Hull_0(g_0 \cdot u). 
	\]
Therefore, it suffices to find a constant $D_0' > 0$ such that for all $u \in \C(\mathring{S})$ with $t_u = 0$, any geodesic segment
	\[ 
		[z',w'] \subset  \widetilde X_0 \cap \Hull_0(u) 
	\]
has $d_0$--length at most $D_0'$.  This follows from the next claim.

\begin{claim}
There exists $D_0'> 0$ such that for any $u \in \C(\mathring{S})$ with $t_u = 0$, the $d_0$--diameter of $\widetilde X_0 \cap \Hull_0(u)$ is at most $D_0'$.
\end{claim}
\begin{proof}[Proof of claim] Since $X_0$ is a compact manifold, $\pi_1(X_0)$ is finitely generated, and hence the image in $G_0$ is a finitely generated subgroup $G_1 < G_0$.  
A conjugate of $G_1$ acts cocompactly on $\widetilde X_0$, and without loss of generality, assume it is $G_1$ itself.  It follows that there exists $r > 0$ such that $N_r(C_0(G_1))$, the $r$--neighborhood of the $d_0$--convex core of $G_1$, contains $X_0$.  
Consequently we have
	\[ 
		\widetilde X_0 \subset N_r(\Hull_0(G_1)).
	\]
By Proposition~\ref{P:bounded diameter fuchsian}, since $G_1$ is finitely generated and purely pseudo-Anosov there exists $D_0' >0$ such that $N_r(\Hull_0(G_1)) \cap \Hull_0(u)$ has diameter at most $D_0'$.  Therefore, so does $\widetilde X_0 \cap \Hull_0(u)$.
\end{proof}
This completes the proof. 
\end{proof}

\begin{proof}[Proof of Proposition \ref{P:bounded diameter 1}]
Let $D' >0$ be as in Lemma \ref{L:bounded diameter 2} and set $D = 2R + KD' + C$, where $K,C,R$ are as in Proposition \ref{P:qi subsurfaces}.
Now suppose $u \subset \C(\mathring{S})$ is a simplex and
	\[ 
		z,w \in N_1(\Hull(G)) \cap \Hull(u).
	\] 
Since this is the intersection of two convex sets in $\HH^3$, the geodesic $[z,w]$ also lies in this intersection.  
By Proposition \ref{P:qi subsurfaces} there exists a $d_{t_u}$--geodesic $[z',w'] \subset \Hull_{t_u}(u)$ with $d$--Hausdorff distance at most $R$ from $[z,w]$.

It follows that
	\[ 
		[z',w'] \subset N_{R+1}(\Hull(G)) \cap \Hull_{t_u}(u),
	\]
and hence has $d_{t_u}$--length at most $D'$ by Lemma \ref{L:bounded diameter 2}.  
Combining this with Proposition \ref{P:qi subsurfaces} and the triangle inequality we see that the $d$--distance between $z,w$ is at most
		\begin{align*}
			d(z,w) 	& \leq 	d(z,z') + d(z',w') + d(w',w) \\
					& \leq 	R + K d_{t_u}(z',w') + C + R\\
					& \leq 	2R + KD'+ C = D
		 \end{align*}
as required.
\end{proof}

\section{End of the proof}

We remark that this is formally just like the proof of Theorem 6.3 in \cite{kentleiningerschleimer}.

\begin{proof}[Proof of Theorem \ref{T:main}]
Fix a vertex $u\in \C^0(\mathring{S})$ and a point $x\in \Hull(u)\cap \Hull(G)$. Let $d_{\C}$ denote the path metric on the curve complex $\C(\mathring{S})$ and equip $G$ with the metric defined by
	\[
		d_G(g,h) \colonequals d_{\Hull(G)}(g\cdot x, h\cdot x) = d(g\cdot x,h\cdot x).
	\]
Since $G$ acts cocompactly on $\Hull(G)$, the \v{S}varc--Milnor lemma (Theorem \ref{T:svarc-milnor}) implies that $d_G$ is quasiisometric to any (finitely generated) word metric on $G$.

We need to show that the orbit map $(G,d_G)\to (\C(\mathring{S}),d_{\C})$ defined by $g\mapsto g\cdot u$ is a quasiisometric embedding, so we must find constants $K\geq 1$ and $C\geq 0$ such that
	\[ 
		\frac{1}{K}d_G(1,g) - C\leq  d_{\C}(u,g\cdot u) \leq K d_G(1,g) + C
	\]
for all $g\in G$. 
Such an upper bound follows immediately from the triangle inequality and the fact that $d_G$ is quasiisometric to the word metric on $G$, and we therefore focus on the lower bound.

Let $(u_0,u_1,\dotsc,u_n)$ be a geodesic path in $\C(\mathring{S})$ from $u = u_0$ to $g\cdot u = u_n$, where $n = d_\C(u,g\cdot u)$. 
We will use this to construct a path from $x$ to $g\cdot x$ in $\Hull(G)$ whose length is bounded in terms of $n$. 
Let $\omega_i = [u_{i-1},u_i]\subset \C(\mathring{S})$, $1\leq i \leq n$, denote the $1$--simplices comprising our $\C$--geodesic from $u$ to $g\cdot u$. 
For a simplex $v\subset \C(\mathring{S})$, recall that $\Hull(v)\subset\HH^3$ is defined to be the convex hull of the limit set of the stabilizer $\pi_1(S)_v$ of $v$ in $\pi_1(S) < \Gamma$. 
Since $u_{i-1},u_i\subset \omega_i$, the corresponding stabilizers are related by $\pi_1(S)_{\omega_i} \subset \pi_1(S)_{u_{i-1}} \cap \pi_1(S)_{u_i}$.
So the corresponding hulls satisfy
	\[
		\Hull(\omega_i) \subset \Hull(u_{i-1})\cap \Hull(u_i).
	\]
In particular, for each $1 \leq i < n$ we have that
	\begin{equation}\label{hullcontainments}
		\Hull(\omega_i),\Hull(\omega_{i+1}) \subset \Hull(u_i).
	\end{equation}

We now construct a piecewise geodesic path $\gamma\subset \HH^3$ connecting $x$ to $g\cdot x$ as follows. 
For each $1 \leq i \leq n$ choose any point $x_i\in \Hull(\omega_i)$; we also set $x_0 = x$ and $x_{n+1} = g\cdot x$. 
Recall that, by choice of $x$, we have $x_0 = x\in \Hull(u) = \Hull(u_0)$ and therefore also that $x_{n+1} = g\cdot x \in \Hull(g\cdot u) = \Hull(u_n)$. 
For each $0\leq i \leq n$ we let let $\gamma_i$ denote the $d$--geodesic $[x_{i},x_{i+1}]$; since $\Hull(u_i)$ is convex, equation \eqref{hullcontainments} and the above implies that $\gamma_i\subset \Hull(u_i)$. 
The concatenation $\gamma = \gamma_0\gamma_1\dotsc\gamma_n$ now gives a piecewise geodesic path from $x$ to $g\cdot x$.

The path $\gamma$ may be arbitrarily long and is furthermore not necessarily contained in $\Hull(G)$. 

Let $\tau\colon \HH^3 \to \Hull(G)$ be the closest point projection. It is a well known fact in hyperbolic geometry that $\tau$ is a contraction and that, furthermore, there exists a constant $T > 0$ such that for any $d$--geodesic segment $\sigma$ outside of $N_1(\Hull(G))$, the projection $\tau(\sigma)$ has length at most $l(\tau(\sigma))\leq T$. 
Now, since $\Hull(u_i)\cap N_1(\Hull(G))$ is convex, it cuts $\gamma_i$ into at most three geodesic segments: at most one in $\Hull(u_i)\cap N_1(\Hull(G))$, which, by Proposition~\ref{P:bounded diameter 1}, has length at most $D$, and at most two which are disjoint from $\Hull(u_i)\cap N_1(\Hull(G))$. 
By the contraction properties of $\tau$, it follows that
\[l(\tau(\gamma_i))\leq 2T+D\]
for each $0\leq i\leq n$. Since $\tau(\gamma)$ is a path in $\Hull(G)$ connecting $x$ to $g\cdot x$, we conclude that
	\[
		d_G(1,g) = d_{\Hull(G)}(x,g\cdot x) \leq l(\tau(\gamma)) \leq (2T+D)(n+1).
	\]
Isolating $n = d_{\C}(u,g\cdot u)$, we find that
	\[
	d_{\C}(u,g\cdot u) = n \geq \frac{1}{2T+D}d_G(1,g) - 1.\qedhere
	\]
\end{proof}

%%%%%%%%%%%%%%%%%%%%%%%%%%
\section{Generalizations}
%%%%%%%%%%%%%%%%%%%%%%%%%%

We now modify the proof of Theorem \ref{T:main} to prove Theorem \ref{T:generalized}.
Suppose that $S$ is a closed surface and $\Theta < \Mod(S)$ is a subgroup of $\Mod(S)$.  
Section \ref{sec:groupExtensions} explains that there is an associated $\pi_1(S)$ extension which includes into the Birman Exact Sequence as in Equation (\ref{E:ses inclusions}):
	\[
		\begin{tikzpicture}[>= to, line width = .075em, 
			baseline=(current bounding box.center)]
		\matrix (m) [matrix of math nodes, column sep=1.5em, row sep = 1.5em, 		text height=1.5ex, text depth=0.25ex]
		{
			1 & \pi_1(S)  & \Gamma_\Theta & \Theta & 1 \\
			1 & \pi_1(S)  & \Mod(\mathring{S}) & \Mod(S) & 1 \\
		};
		\path[->,font=\scriptsize]
		(m-1-1) edge 					(m-1-2)
		(m-1-2) edge 					(m-1-3)
		(m-1-3) edge 					(m-1-4)
		(m-1-4) edge 					(m-1-5)
		
		(m-2-1) edge 					(m-2-2)
		(m-2-2) edge 					(m-2-3)
		(m-2-3) edge 					(m-2-4)
		(m-2-4) edge 					(m-2-5)
		
		(m-1-3) edge					(m-2-3)
		(m-1-4) edge					(m-2-4)
		;
		\draw[double distance = .15em,font=\scriptsize]
		(m-2-2) -- 					(m-1-2)
		;
		\end{tikzpicture}
	\]

According to Theorem~\ref{T:ccciffhyperbolic}, the group $\Gamma_\Theta$ is $\delta$--hyperbolic if and only if $\Theta$ is convex cocompact; see also \cite{Mj.Sardar.2012}.

\begin{MainGeneralized}  Suppose $S$ is a closed surface, $\Theta < \Mod(S)$ is a convex cocompact subgroup and $G< \Gamma_\Theta$ is a finitely generated quasiisometrically embedded subgroup.  
If $G$ is purely pseudo-Anosov as a subgroup of $\Mod(\mathring{S})$, then it is convex cocompact.
\end{MainGeneralized}

\begin{remark}
We note that although we are able to replace the $\Gamma$ from Theorem \ref{T:main} with a more general class of groups, we do need the assumption on the subgroup $G < \Gamma_\Theta$.  
This is due to the fact that a generalization of Lemma \ref{L:cccKleinian} seems quite difficult, or perhaps false, in this more general setting.  
\end{remark}

To simplify notation, we henceforth write $\Gamma = \Gamma_\Theta$.

We start by describing a geometric model for $\Gamma$ that will be most useful for the proof.  
By passing to finite index subgroups, we assume $\Theta$ is torsion--free.  
Let $\widetilde B$ denote the Cayley graph of $\Theta$ with respect to some finite generating set, and $B = \widetilde B/\Theta$ the quotient wedge of circles.
Choose any continuous map $\widetilde B \to \T(S)$ which is equivariant with respect to the actions of $\Theta$ on $\widetilde B$ by covering transformations and on $\T(S)$ via the action induced by the inclusion $\Theta \to \Mod(S)$.  
The Bers fibration
	\[
		\begin{tikzpicture}[>= to, line width = .075em, 
			baseline=(current bounding box.center)]
		\matrix (m) [matrix of math nodes, column sep=1.5em, row sep = 1em, 		text height=1.5ex, text depth=0.25ex]
		{
			\HH^2  & \T(\mathring{S}) & \T(S) \\
		};
		\path[->,font=\scriptsize]
		(m-1-1) edge (m-1-2)
		(m-1-2) edge (m-1-3)
		;
		\end{tikzpicture}
	\]
is equivariant with respect to the Birman Exact Sequence, and we can pull back the bundle to $\widetilde B$ so that all maps are equivariant:
	\begin{equation}\label{E:lots of equivariance}
		\begin{tikzpicture}[>= to, line width = .075em, 
			baseline=(current bounding box.center),
			descr/.style={fill=white},
			cross line/.style={preaction={draw=white, -, line width=6pt}}
			]
		\matrix (m) [matrix of math nodes, column sep=.5em, row sep = 1.25em, 		text height=1.5ex, text depth=0.25ex]
		{
			1 & & \pi_1(S) & & \Gamma & & \Theta & & 1 \\
			& & & \HH^2 & & \widetilde X & & \widetilde B & \\
			1 & & \pi_1(S) & & \Mod(\mathring{S}) & & \Mod(S) & & 1 \\
			& & & \HH^2 && \T(\mathring{S}) & & \T(S) & \\
		};
		\path[->,font=\scriptsize]
		(m-1-1) 	edge (m-1-3)
		(m-1-3) 	edge (m-1-5)
		(m-1-5) 	edge (m-1-7)
				edge (m-3-5)
		(m-1-7) 	edge (m-1-9)
				edge (m-3-7)

		(m-3-1) 	edge (m-3-3)
		(m-3-3) 	edge (m-3-5)
		(m-3-5) 	edge (m-3-7)
		(m-3-7) 	edge (m-3-9)
%white edges:
		(m-2-4) 	edge [-,line width=.5em,draw=white] 	(m-2-6)
		(m-2-6) 	edge [-,line width=.5em,draw=white]	(m-2-8)
				edge [-,line width=.5em,draw=white]	(m-4-6)
		(m-2-8) 	edge [-,line width=.5em,draw=white]	(m-4-8)
		
		(m-2-4) 	edge	[-,line width=.75em,draw=white]	(m-4-4)
%edges in front:
		(m-2-4) 	edge 					(m-2-6)
		(m-2-6) 	edge node[above,near start] {$\varphi$} (m-2-8)
				edge 					(m-4-6)
		(m-2-8) 	edge 					(m-4-8)
		
		(m-4-4) 	edge 					(m-4-6)
		(m-4-6) 	edge 					(m-4-8)
		;
		\draw[double distance = .15em,font=\scriptsize]
		(m-1-3) --	(m-3-3)
		(m-2-4) --	(m-4-4)
		;
		\path[,distance = .15em]
		(m-1-3) --  node[pos=.45]{$\circlearrowleft$}	(m-2-4)
		(m-1-5) --  node[pos=.45]{$\circlearrowleft$}	(m-2-6)
		(m-1-7) --  node[pos=.45]{$\circlearrowleft$}	(m-2-8)
		
		(m-3-3) -- node{$\circlearrowleft$}	(m-4-4)
		(m-3-5) -- node{$\circlearrowleft$}	(m-4-6)
		(m-3-7) -- node{$\circlearrowleft$}	(m-4-8)
		;
		\end{tikzpicture}
	\end{equation}

We give $\widetilde X$ a $\Gamma$--invariant geodesic metric $d$ for which the induced path metric on the fiber $\varphi^{-1}(t) = \HH^2_t$ for $t \in \widetilde B$ is the hyperbolic metric $d_t$.  
Each quotient by the corresponding group is compact, and this produces an $S$--bundle over $B$:
	\[
		\begin{tikzpicture}[>= to, line width = .075em, 
			baseline=(current bounding box.center)]
		\matrix (m) [matrix of math nodes, column sep=1.5em, row sep = 1em, 		text height=1.5ex, text depth=0.25ex]
		{
			S  & X & B. \\
		};
		\path[->,font=\scriptsize]
		(m-1-1) edge (m-1-2)
		(m-1-2) edge (m-1-3)
		;
		\end{tikzpicture}
	\]
By the \v{S}varc--Milnor lemma (Theorem \ref{T:svarc-milnor}), any orbit map $\Gamma \to \widetilde X$ is a quasiisometry with respect to the word metric on $\Gamma$ for any fixed finite generating set.

The analogue of Proposition \ref{P:qi subsurfaces} we need is the following.   
Here $\Gamma$ acts on $\widetilde B$ via the homomorphism $\Gamma \to \Theta$ and on $\C(\mathring{S})$ by the homomorphism $\Gamma \to \Mod(\mathring{S})$.    
We write $\Hull_t(u)$ to denote the convex hull in $\HH_t^2$ of the stabilizer $\pi_1(S)_u$ of $u$ in $\pi_1(S)$, as before.

\begin{proposition} \label{P:qi subsurfaces general}
Suppose that we are in the situation of Diagram \eqref{E:lots of equivariance} and $\widetilde X$ is $\delta$--hyperbolic.  
There exist $K,C > 0$ with the following property.  
For any simplex $u \subset \C(\mathring{S})$ there exists a vertex $t_u  \in \widetilde B^{(0)}$ satisfying $t_{\gamma \cdot u} = \gamma \cdot t_u$ for all $\gamma \in \Gamma$, and such that the inclusion
	\[ 
	\Hull_{t_u}(u) \to \widetilde X 
	\]
is a $(K,C)$--quasiisometric embedding.
\end{proposition}

In Section \ref{S:Scott.Swarup}, we derive Theorem \ref{T:General.Scott.Swarup} from this Proposition.

The proof of Proposition \ref{P:qi subsurfaces general} requires the following analogue of Proposition \ref{P:widthmain} which is proven in \cite{kentleiningerwidth}.  
Given any simplex $u \subset \C(\mathring{S})$, we let $\Hull(u)$ denote the union of all quasiinvariant geodesic axes in $\widetilde X$ of elements in $\pi_1(S)_u$.
\begin{proposition} \label{P:coarse width}
Suppose that we are in the situation of Diagram \eqref{E:lots of equivariance} and $\widetilde X$ is $\delta$--hyperbolic.  
Then there exists $W > 0$ such that for any simplex $u \subset \C(\mathring{S})$, the set $\Hull(u)$ has $\diam_{\widetilde{B}}(\varphi(\Hull(u))) < W$.
\qed
\end{proposition}

With this Proposition, the proof of Proposition \ref{P:qi subsurfaces general} is similar to that of Proposition \ref{P:qi subsurfaces}. The key idea is again to pick $t_u \in \widetilde B^{(0)}$ lying within a uniformly bounded distance of the image of $\Hull(u)$ in $\widetilde B$ so that $\varphi^{-1}(B_R(t_u))$ will play the role of $\HH^2_{t_u} \times [t_u-W,t_u+W]$. As before, we will see that the inclusion $\HH^2_{t_u}\to \varphi^{-1}(B_R(t_u))$ is uniformly bilipschitz and, using the fact that $\Hull(u)\subset\varphi^{-1}(B_R(t_u))$ is convex, we will show that distances in $\Hull_{t_u}(u)$ are comparable to those in $\Hull(u)$ and thus also in $\widetilde{X}$.

We also need the following minor modification of Lemma \ref{L:segments almost extend}.
\begin{lemma}  \label{L:segments almost extend 2}
There exists $\delta' > 0$ such that for any convex cocompact Fuchsian group $H < \PSL(2,\RR)$, any geodesic segment $[z,w]$ in $\Hull(\Lambda_H) \subset \HH^2$ has Hausdorff distance at most $\delta'$ from a geodesic segment $[z_0,w_0]$ which is contained in a biinfinite \textit{periodic} geodesic in $\Hull(\Lambda_H)$.
\end{lemma}
\begin{proof}
Since the fixed points of hyperbolic elements is dense in $\Lambda_H  \times \Lambda_H$, it follows that any biinfinite geodesic in $\Hull(\Lambda_H)$ is a limit of periodic geodesics.  
So we may apply Lemma \ref{L:segments almost extend} to first find a segment $[z_0',w_0']$ in some biinfinite geodesic in $\Hull(\Lambda_H)$, then approximate this as close as we like by a segment $[z_0,w_0]$ contained in a periodic geodesic.
\end{proof}

\begin{proof}[Proof of Proposition \ref{P:qi subsurfaces general}]
For any simplex $u \subset \C(\mathring{S})$ we choose a vertex $t_u$ within a  distance at most $1$ from $\varphi(\Hull(u))$, subject to the equivariance condition $t_{\gamma \cdot u} = \gamma \cdot t_u$ (compare the proof of Proposition \ref{P:qi subsurfaces}).  
We now prove that $t_u$ has the required properties.

Let $R = W + 1$, let $t \in \widetilde B^{(0)}$, and consider the preimage $\varphi^{-1}(B_R(t))$ of the closed ball $B_R(t)$.
Equip $\varphi^{-1}(B_R(t))$ with the induced path metric $d^\times_t$.
Since $\pi_1(S)$ acts cocompactly on $\HH^2_t$ and $\varphi^{-1}(B_R(t))$, the \v{S}varc--Milnor lemma (Theorem \ref{T:svarc-milnor}) implies that this inclusion is a $(K,C')$--quasiisometry for some  $K,C' > 1$.  
In particular, the space $\varphi^{-1}(B_R(t))$ is Gromov hyperbolic.
Since $\Theta$ acts transitively by isometries on $\widetilde B^{(0)}$ (because $B$ has only one vertex), $K$ and $C'$ are independent of $t$.  
Given any simplex $u \subset \C(\mathring{S})$, we have $\Hull_{t_u}(u) \subset \varphi^{-1}(B_{R}(t_u))$ by assumption.  Since the inclusion $\Hull_{t_u}(u) \to \HH^2_{t_u}$ is an isometric embedding, the inclusion
	\[ 
		\Hull_{t_u}(u) \to \varphi^{-1}(B_R(t_u))
	\]
is a $(K,C')$--quasiisometric embedding.

Let $R' > 0$ be the stability constant for $(K,C')$-quasigeodesics in the Gromov hyperbolic metric space $\varphi^{-1}(B_R(t_u))$ given by Theorem \ref{T:QIstability}.  
Let $\delta' > 0$ be the constant from Lemma \ref{L:segments almost extend 2} and set
	\[
		C = 4K(\delta' + R') + C'.
	\]
The next claim will prove the proposition.

\begin{claim} For any $u \subset \C(\mathring{S})$ the inclusion
	\[ 
		\Hull_{t_u}(u) \to \widetilde X 
	\]
is a $(K,C)$--quasiisometric embedding.
\end{claim}
\begin{proof}[Proof of Claim]
Let $z',w' \in \Hull_{t_u}(u)$ be any two points.  
Observe that $\pi_1(S)_u$ is a finitely generated subgroup of the closed surface group $\pi_1(S)$, so it is a convex cocompact Fuchsian group.  
By Lemma \ref{L:segments almost extend 2} there are points $z'_0,w'_0 \in \Hull_{t_u}(u)$ such that
	\begin{equation}\label{E:close in hull t again}
		d_{t_u}(z',z'_0), d_{t_u}(w',w'_0) \leq \delta'
	\end{equation}
and such that the geodesic segment $[z'_0,w'_0]$ extends to a biinfinite periodic $d_{t_u}$--geodesic $\mathcal G' \subset \Hull_{t_u}(u)$ invariant under an element $h\in \pi_1(S)_u$.

Let $\mathcal G$ be a quasiinvariant $d$--geodesic axis for $h$; thus $\mathcal G\subset \Hull(u) \subset \varphi^{-1}(B_{R}(t_u))$. In particular, $\mathcal G$ is a $d^\times_{t_u}$--geodesic.  Since the $d_{t_u}$--geodesic $\mathcal G'$ is a $(K,C')$--quasigeodesic with respect to $d^\times_{t_u}$, the $d_{t_u}^\times$--Hausdorff distance between $\mathcal G$ and $\mathcal G'$ is at most $R'$.

Let $z,w \in \mathcal G$ be points with
	\[ 
		d^\times_{t_u}(z_0',z), d^\times_{t_u}(w_0',w) \leq R'.
	\]
Notice that $d(z,w) = d^\times_{t_u}(z,w)$ because $z,w\in \mathcal G \subset \Hull(u)$.
Appealing to this and \eqref{E:close in hull t again}, we have
	\begin{equation} \label{E:z close in XX} 
		d(z',z)   \leq  d^\times_{t_u}(z',z) \leq  d^\times_{t_u}(z',z_0') + d^\times_{t_u}(z_0',z)  \leq \delta'  + R',
	\end{equation}
since $d$ is less than $d^\times_{t_u}$, 
and likewise
	\begin{equation} \label{E:w close in XX} 
		d(w',w)  \leq  d^\times_{t_u}(w',w) \leq \delta' + R'.
	\end{equation}

Combining \eqref{E:z close in XX} and \eqref{E:w close in XX} with the triangle inequality, the fact that $\Hull_{t_u}(u) \to \varphi^{-1}(B_R(t_u))$ is a $(K,C')$--quasiisometric embedding, and the fact that $d(z,w) = d^\times_{t_u}(z,w)$, we find
	\begin{align*}
		d_{t_u} (z',w') 	& \leq  K d^\times_{t_u}(z',w') + C'\\
					& \leq  K(d^\times_{t_u}(z,w) + d^\times_{t_u}(z',z) + d^\times_{t_u}(w,w')) + C'\\
					& \leq  K d(z,w) + 2K(\delta' + R') + C' \\
					& \leq   K(d(z',w') + d(z,z') + d(w',w)) + 2K(\delta' + R') + C'\\
					& \leq  K d(z',w') + 2K(\delta'  + R') + 2K(\delta' + R') + C'\\
					& =  K d(z',w') + C.
\end{align*}
On the other hand, since $d \leq d^\times_{t_u} \leq d_{t_u}$ on $\Hull_{t_u}(u)$, it follows that
	\[ 
		d(z',w') \leq d^\times_{t_u}(z',w') \leq d_{t_u}(z',w') \leq K d_{t_u}(z',w') + C.
	\]
Therefore, the inclusion $\Hull_{t_u}(u) \to \widetilde X$ is a $(K,C)$--quasiisometric embedding, proving the claim.
\end{proof}
This claim completes the proof of the proposition.
\end{proof}

We have the following corollary which will be used to prove Theorem \ref{T:generalized}.

\begin{corollary} \label{C:close hulls}
Suppose that we are in the situation of Diagram \eqref{E:lots of equivariance} and $\widetilde X$ is $\delta$--hyperbolic.  
Then there exists $D_0 > 0$ such that for any pair of adjacent vertices $u_1,u_2 \in \C^{(0)}(\mathring{S})$, there are points $x' \in \Hull_{t_{u_1}}([u_1,u_2])$ and  $x'' \in \Hull_{t_{u_2}}([u_1,u_2])$ with $d(x',x'') < D_0$.
\end{corollary}
\begin{proof} 
First observe that the inclusions
	\[ 
		\Hull_{t_{u_1}}([u_1,u_2]) \subset \Hull_{t_{u_1}}(u_1) 
			\mbox{\quad and \quad} \Hull_{t_{u_2}}([u_1,u_2]) \subset \Hull_{t_{u_2}}(u_2) 
	\]
are isometric embeddings, and hence $\Hull_{t_{u_1}}([u_1,u_2])$ and $\Hull_{t_{u_2}}([u_1,u_2])$ are $(K,C)$--quasiisometrically embedded in $\widetilde X$.  
A quasiinvariant $d$--geodesic axis for any element of $\pi_1(S)_{[u_1,u_2]}$ is contained in a uniformly bounded neighborhood of each of $\Hull_{t_{u_1}}([u_1,u_2])$ and $\Hull_{t_{u_2}}([u_1,u_2])$, and hence there are points in these hulls within some uniform distance $D_0 > 0$ of each other.
\end{proof}

Let $G_0 = G \cap \pi_1(S)$ and $\hat G < \Theta$ denote the image of $G$ under the homomorphism $\Gamma \to \Theta$.  
Denote the quotients of $\widetilde X$ by the actions of $G_0$, $G$, and $\pi_1(S)$ by $X_{G_0}$, $X_G$, and $X_S$, respectively. 
Denote the quotient of $\HH^2$ by the action of $G_0$ by $S_{G_0}$, and the quotient of $\widetilde B$ by the action of $\hat G$ by $B_G$.
We arrange all these quotient maps and all previous maps into the following diagram, labeling those we will need to refer to explicitly.
	\begin{equation}\label{E:bigdiagramgeneral}
		\begin{tikzpicture}[>= to, line width = .075em, 
			baseline=(current bounding box.center),
			descr/.style={fill=white},
			cross line/.style={preaction={draw=white, -, line width=6pt}}
			]
		\matrix (m) [matrix of math nodes, column sep=1.5em, row sep = 2em, 		text height=1.5ex, text depth=0.25ex]
		{
			& \HH^2 
			& S_{G_0} 
			& 
			& 
				\\
			& 
			& S 
			& 
			& 
				\\
			\HH^2 \times \widetilde{B}
			& S_{G_0} \times \widetilde{B}
			& S \times  \widetilde{B}
			& X_G 
			& 
				\\	
			\widetilde{X} 
			& 
			& 
			& 
			& 
				\\			
			& X_{G_0}  
			& X_S 
			&   
			&  X 
				\\
			& & & & \\
			& \widetilde{B}
			&                                     
			& B_G
			& B 
				\\					
		};
		\draw[double distance = .15em,font=\scriptsize]
		(m-3-1) --	(m-4-1)
		(m-3-2) --	(m-5-2)
		(m-3-3) --	(m-5-3)
		;
		\path[->,font=\scriptsize]
		(m-1-2) 	edge		(m-1-3)
				edge		(m-3-1.north)	
		(m-1-3)	edge		(m-2-3)
				edge		(m-3-2.north)
				edge		(m-3-4)
		(m-2-3)	edge		(m-3-3)
		(m-3-1)	edge		(m-3-2)
		(m-3-2)	edge		(m-3-3)
		
		(m-5-3) 	edge		(m-5-5)
				
%white edges:
		(m-3-4)	edge [-,line width=.5em,draw=white] 				(m-7-4)
		(m-4-1) 	edge [-,line width=.5em,draw=white,bend right = 8] 	(m-3-4)
		(m-5-2) 	edge [-,line width=.5em,draw=white] 				(m-3-4)
%edges in front:
		(m-3-4)	edge		node[descr,near start] {$f$}		(m-7-4)
		(m-4-1) 	edge [bend right = 8]	node[descr] {$p$}		(m-3-4)
				edge			node[descr,midway] {$p_1$}	(m-5-2)
				edge									(m-7-2)
		(m-5-2) 	edge		node[descr,near start] {$p_0$}		(m-3-4)
				edge		node[descr] {$\widetilde{f}$}			(m-7-2)
				edge									(m-5-3)
		(m-5-3)	edge									(m-7-2)
		(m-5-5)	edge									(m-7-5)
		(m-7-4)	edge									(m-7-5)
		(m-7-2)	edge		node[descr] {$q$}				(m-7-4)

		;		
		\end{tikzpicture}
	\end{equation}

The fact that each of the spaces $\widetilde X$, $X_{G_0}$, and $X_S$ are products follows from the fact that $X_S$ is a product, which in turn follows from the fact that the quotient of $\T(\mathring{S})$ by $\pi_1(S)$ is an $S$--bundle over the contractible space $\T(S)$.  
We also note that $X_{G_0}$ and $X_G$ are $S_{G_0}$--bundles over $\widetilde B$ and $B_G$, respectively, with the latter bundle the quotient of the former by the action of $G/G_0 = \hat G$ by bundle transformations.

For any vertex $t \in \widetilde B^{(0)}$ we let $\HH_t^2$, $S_{G_0,t}$, and $S_t$ denote the fibers over $t$ in $\widetilde X$, $X_{G_0}$, and $X_S$, respectively.  
We equip these with their induced path metrics, all of which we denote $d_t$.  
Similarly, given $v \in B_G^{(0)}$, let $S_{G_0,v}$ denote the fiber over $v$ in $X_G$ with its path metric $d_v$. 
 Observe that if $q(t) = v$, then the restriction of $p_0$ to $(S_{G_0,t},d_t)$ is an isometry to $(S_{G_0,v},d_v)$.

\begin{lemma} \label{L:bounded intersection general}
Suppose that we are in the situation of Diagram \eqref{E:lots of equivariance}, the space $\widetilde X$ is $\delta$--hyperbolic,  and $G < \Gamma$.  
Given $x \in \widetilde X$ and $R > 0$ there exists $D' > 0$ with the following property.  
If $u$ is a simplex of $\C(\mathring{S})$, then any $d_{t_u}$--geodesic segment
	\[ 
		[z,w]_{t_u} \subset N_R(G \cdot x) \cap \Hull_{t_u}(u) 
	\]
has length at most $D'$.
\end{lemma}
\begin{proof}  The reader may wish to refer to Diagram (\ref{E:bigdiagramgeneral}) throughout the proof.

For any $x \in \widetilde X$ and $R > 0$ we consider the closed $R$--neighborhood $N_R(G \cdot x)$ of the $G$--orbit of $x$.  
Observe that $p(G \cdot x) = p(x)$, and so $p(N_R(G \cdot x)) = B_R(p(x))$, the closed ball of radius $R$ about $p(x)$, which is compact.
Since $f$ is continuous, it follows that $f(B_R(p(x))) \subset B_G$ is compact, and hence contains only finitely many vertices
	\[ 
		V =  \{ v_1,\ldots,v_n \} =  f(B_R(p(x))) \cap B_G^{(0)} .
	\]
Also, for each $i = 1,\ldots,n$, pick $t_i$ with $q(t_i) = v_i$ and set
	\[ 
	T = \{t_1,\ldots,t_n \} \subset \widetilde B^{(0)}
	\]
so that $G \cdot T = q^{-1}(V)$ (where $G$ is acting on $T \subset \widetilde B$ by the quotient $G \to \hat G$).
Then, given any $t \in \widetilde B^{(0)}$ we have
	\begin{equation}\label{E:nonempty intersection condition} 
	\begin{aligned} 
			\HH_t^2 \cap  N_R(G \cdot x) \neq \emptyset 	
				& \quad \Leftrightarrow \quad S_{G_0,t} \cap p_1(N_R(G \cdot x))  \neq \emptyset 
				 \\
				&   \quad \Leftrightarrow \quad  \exists \, g \in G \mbox{ such that } g \cdot t \in T . 
 	\end{aligned} 
	\end{equation}

For each $i = 1,\ldots,n$, the map $p_0$ takes $(S_{G_0,t_i},d_{t_i})$ isometrically to $(S_{G_0,v_i},d_{v_i})$, and restricts to a homeomorphism on the intersections
	\[
		\begin{tikzpicture}[>= to, line width = .075em, 
			baseline=(current bounding box.center)]
		\matrix (m) [matrix of math nodes, column sep=1.5em, row sep = 1em, 		text height=1.5ex, text depth=0.25ex]
		{
			S_{G_0,t_i} \cap p_1(N_R(G \cdot x))
			& & S_{G_0,v_i} \cap B_R(p(x)). 						\\
		};
		\path[->,font=\scriptsize]
		(m-1-1) edge 	node[below]{$p_0$}	node[above]{$\cong$}	(m-1-3)
		;
		\end{tikzpicture}
	\]
Since the target of this restriction is compact, so is the domain.  So there is a compact connected subsurface $\Sigma_i \subset S_{G_0,t_i}$ with
	\[ 
	S_{G_0,t_i} \cap p_1(N_R(G \cdot x)) \subset \Sigma_i. 
	\]
We may assume $\Sigma_i$ is $\pi_1$--injective, and we let $G_i = \pi_1(\Sigma_i) < G_0$ be the \textit{finitely generated} image.

Let $\widetilde \Sigma_i$ denote the component of $p_1^{-1}(\Sigma_i) \subset \HH_{t_i}^2$ stabilized by $G_i$ and let $r_i > 0$ be such that
	\[ 
		\widetilde \Sigma_i \subset N_{r_i}(\Hull_{t_i}(G_i)). 
	\]
Observe that any geodesic segment $[z,w]_{t_i} \subset \HH_{t_i}^2 \cap N_R(G \cdot x)$ projects by $p_1$ to be contained in $S_{G_0,t_i} \cap p_1(N_R(G \cdot x)) \subset \Sigma_i$.  
Therefore $[z,w]_{t_i}$ is contained in a $G_0$ translate of $\widetilde \Sigma_i$, and hence a $G_0$--translate of $N_{r_i}(\Hull_{t_i}(G_i))$.

Now, let $u \subset \C(\mathring{S})$ denote any simplex.  Given any geodesic segment in the intersection
	\[ 
		[z,w]_{t_u} \subset \Hull_{t_u}(u) \cap N_R(G \cdot x), 
	\]
we may apply an element $g \in G$ with $g \cdot t_u = t_i$ to this, by (\ref{E:nonempty intersection condition}).  
By Proposition \ref{P:qi subsurfaces general}, we have $g \cdot t_u = t_{g \cdot u}$, and so
	\[ 
	[g \cdot z,g \cdot w]_{t_{g \cdot u}} 
		= [g \cdot z, g \cdot w]_{t_i}  
		\subset \Hull_{t_i}(g \cdot u) \cap N_R(G \cdot x) 
		= \Hull_{t_{g \cdot u}}(g \cdot u) \cap N_R(G \cdot x).
	\]
Since $g$ acts by isometries, it follows that
	\[ 
		\diam([z,w]_{t_u}) = \diam([g \cdot z,g\cdot w]_{t_{g \cdot u}}).
	\]
So it suffices to prove the lemma for segments $[z,w]_{t_u}$ where $t_u  = t_i$ for some $i$.  
As noted above, all such segments are contained in a $G_0$ translate of $N_{r_i}(\Hull_{t_i}(G_i))$.  
Therefore, appealing to Proposition \ref{P:qi subsurfaces general} again, it suffices to prove the lemma for segments
	\[ 
	[z,w]_{t_i} \subset N_{r_i}(\Hull_{t_i}(G_i)) \cap \Hull_{t_u}(u) 
	\]
where $t_u = t_i$.

By Proposition~\ref{P:bounded diameter fuchsian}, we again see that there exists a $D_i$ which bounds the length of such a segment, depending on $G_i$ and $r_i$, but not $u$.
Setting $D' = \max \{D_1,\ldots,D_n\}$ completes the proof.
\end{proof}

Given a subset $Y \subset \widetilde X$, let $\pi \co \widetilde X \to Y$ denote a closest point projection map.  
The following is a consequence of quasiconvexity and hyperbolicity.

\begin{lemma} \label{L:lipschitz projection}
Suppose that we are in the situation of Diagram \eqref{E:lots of equivariance}, the space $\widetilde X$ is $\delta$--hyperbolic, and $Y \subset \widetilde X$ an $A$--quasiconvex subset.  
There is a $\lambda > 0$ such that $\pi$ is $(\lambda,\lambda)$--coarsely Lipschitz.

Moreover, given $K,C > 0$ there exists $R_0 >0$ such that for any $R> R_0$ and any $(K,C)$--quasigeodesic $\gamma \subset \widetilde X$, 
	\[
		\diam(\pi(\gamma)) \leq \diam((\gamma \cap N_R(Y))_0) + R_0
	\]
where $(\gamma \cap N_R(Y))_0$ is the longest segment of $\gamma$ contained in $N_R(Y)$.
\end{lemma}
\begin{proof}
The first part is well--known.   
To prove the second part, we observe that a quasigeodesic $\gamma = [z,w]$ can be decomposed into three segments $[z,w] = [z,z'][z',w'][w',w]$ (some of which may be empty), where $[z',w']$ remains a bounded distance from $Y$ and $[z,z']$ remains a bounded distance from a geodesic joining $z$ to $\pi(z)$ (and therefore has uniformly bounded projection diameter).
Similarly, the path $[w',w]$ remains a bounded distance from a geodesic joining $\pi(w)$ to $w$.
\end{proof}

The next proposition follows by simply assembling the results above.
\begin{proposition} \label{P:bounded diameter projection general}  
Suppose that we are in the situation of Diagram \eqref{E:lots of equivariance}, the space $\widetilde X$ is $\delta$--hyperbolic,  and $G < \Gamma$ quasiisometrically embedded.  Given $x \in \widetilde X$ and letting $\pi \co \widetilde X \to G \cdot x$ denote a closest point projection map, there exists $D > 0$ with the following property.  
Given any simplex $u \subset \C(\mathring{S})$ and $t_u$--geodesic segment $[z,w]_{t_u} \subset \Hull_{t_u}(u)$ we have $\diam(\pi([z,w]_{t_u})) < D$.
\end{proposition}
\begin{proof}
Let $K,C$ be as in Proposition \ref{P:qi subsurfaces general}.  Since $G < \Gamma$ is quasiisometrically embedded, $G \cdot x$ is also, and hence is $A$--quasiconvex for some $A >0$.  
Let $R_0 > 0$ be given by Lemma \ref{L:lipschitz projection}, and fix $R \geq R_0$.  
Finally,  let $D' > 0$ be as given by Lemma \ref{L:bounded intersection general} and set $D = D' + R_0$.  

Now $[z,w]_{t_u} \subset \Hull_{t_u}(u)$ is $(K,C)$--quasigeodesic by Proposition \ref{P:qi subsurfaces general}.  By Lemma \ref{L:bounded intersection general} and the second part of Lemma \ref{L:lipschitz projection}, it follows that 
	\[ 
		\diam(\pi([z,w]_{t_u})) \leq \diam(([z,w]_{t_u} \cap N_R(G \cdot x))_0) + R_0  \leq D' + R_0 = D,
	\]
where $([z,w]_{t_u} \cap N_R(G \cdot x))_0$ is the longest segment in the intersection
	\[ 
		[z,w]_{t_u} \cap N_R(G \cdot x)  \subset \Hull_{t_u}(u) \cap N_R(G \cdot x).
	\]
This completes the proof.
\end{proof}

\begin{proof} [Proof of Theorem \ref{T:generalized}]
This now follows a similar outline to the proof of Theorem \ref{T:main}.

We choose any vertex $u \in \C(\mathring{S})$ and let $x \in \Hull_{t_u}(u) \subset \widetilde X$ be any point. 
Let $D > 0$ be as in Proposition \ref{P:bounded diameter projection general}.
Since $G < \Gamma$  is quasiisometrically embedded, distances in $G$ are comparable to those in $G \cdot x \subset \widetilde X$.  As in the proof of Theorem \ref{T:main}, it suffices to prove that there exist constants $K_0,C_0>$ such that 
	\[ 
		d(x,g \cdot x) \leq K_0 d_\C(u,g \cdot u) + C_0.
	\]
Let $D_0  > 0$ be as in Corollary \ref{C:close hulls} and $\lambda > 0$ as in Lemma \ref{L:lipschitz projection}.  
We claim that $C_0 = \max \{ \lambda(D_0+ 1),D \}$ and $K_0 = 2C_0$ suffices.\\

Let $u = u_0,\ldots,u_n = g \cdot u$ denote the vertices of a geodesic $[u,g \cdot u] \subset \C(\mathring{S})$ connecting $u$ to $g \cdot u$, so that $n = d_\C(u,g \cdot u)$.
For each $1 \leq i \leq n$, choose points
	\[ 
		x_i' \in \Hull_{t_{u_{i-1}}}([u_{i-1},u_i]) \mbox{ and } x_i'' \in \Hull_{t_{u_i}}([u_{i-1},u_i]) 
	\]
which are a distance $D_0$ apart, which is possible by Corollary \ref{C:close hulls}.

Now consider the path $\gamma$ connecting $x$ and $g \cdot x$  given by
	\[ 
	\gamma = [x,x_1']_{t_{u_0}} [x_1',x_1''] [x_1'',x_2']_{t_{u_1}} [x_2',x_2''] \cdots [x_{n-1}'',x_n']_{t_{u_{n-1}}} [x_n',x_n''][x_n'',g \cdot x]_{t_{u_n}}. 
	\]
Here $[z,w]$ denotes a $d$--geodesic from $z$ to $w$ in $\widetilde X$ and $[z,w]_t$ denotes a $d_t$--geodesic from $z$ to $w$ in $\HH^2_t$.  Since
	\[ 
	x_i'' \in \Hull_{t_{u_i}}([u_{i-1},u_i]) \subset \Hull_{t_{u_i}}(u_i)
	\]
and
	\[ 
	x_{i+1}' \in \Hull_{t_{u_i}}([u_i,u_{i+1}]) \subset \Hull_{t_{u_i}}(u_i), 
	\]
it follows that $[x_i'',x_{i+1}']_{t_{u_i}} \subset \Hull_{t_{u_i}}(u_i)$ for every $1 \leq i \leq n-1$.  In particular, the path $\gamma$ alternates between geodesic segments in hulls $\Hull_{t_{u_i}}(u_i)$ and segments of the form $[x_i',x_i'']$ (note that $x,x_1' \in \Hull_{t_{0}}(u_0)$ and $x_n'',g \cdot x  \in \Hull_{t_{u_n}}(u_n)$).

Let $\pi \co \widetilde X \to G \cdot x$ denote a closest point projection.
By Proposition \ref{P:bounded diameter projection general}, for every $i = 0,\ldots,n-1$ we have
	\[ 
		d(\pi(x_i''),\pi(x_{i+1}')) \leq \diam( \pi([x_i'',x_{i+1}']_{t_{u_i}})) \leq D 
	\]
and
	\[ 
		d(x,\pi(x_1')), d(\pi(x_n''),g \cdot x) \leq D.
	\]
Since $\pi$ is $(\lambda,\lambda)$--coarsely Lipschitz, we also have
	\[ 
		d(\pi(x_i'),\pi(x_i'')) \leq \lambda d(x_i',x_i'') + \lambda \leq \lambda (D_0 + 1). 
	\]

Therefore, since $C_0 = \max \{ \lambda(D_0 + 1),D \}$, $K_0 = 2C_0$, and $d_\C(u, g \cdot u) = n$ we have
	\begin{align*} 
		d(x,g \cdot x) 	& \leq d(x,\pi(x_1')) 
						 +  \sum_{i=1}^n d(\pi(x_i'),\pi(x_i'')) \\
					 &   \quad \quad \quad \quad \quad \ \ \
					  	+ \sum_{i=1}^{n-1} d(\pi(x_i''),\pi(x_{i+1}'))  + d(\pi(x_n''),g \cdot x) \\
					 & \leq  (n+1)D + n (\lambda(D_0 + 1)) \\
					 & \leq  (2n+1)C_0\\
					 & =  (2d_\C(u,g \cdot u) + 1)C_0\\
					 & =  K_0 d_\C(u,g \cdot u) + C_0
	\end{align*}
as required.
\end{proof}

%%%%%%%%%%%%%%%%%%%%%%%%%%
\section{On a theorem of Scott and Swarup}\label{S:Scott.Swarup}
%%%%%%%%%%%%%%%%%%%%%%%%%%

We now prove our generalization of Scott and Swarup's Theorem \cite{Scott.Swarup.1990}.
\begin{ScottSwarupGeneralized}
Let 
	\[
		\begin{tikzpicture}[>= to, line width = .075em, 
			baseline=(current bounding box.center)]
		\matrix (m) [matrix of math nodes, column sep=1.5em, row sep = 1em, 		text height=1.5ex, text depth=0.25ex]
		{
			1 & \pi_1(S)  & \Gamma_\Theta & \Theta & 1 \\
		};
		\path[->,font=\scriptsize]
		(m-1-1) edge (m-1-2)
		(m-1-2) edge (m-1-3)
		(m-1-3) edge (m-1-4)
		(m-1-4) edge (m-1-5)
		;
		\end{tikzpicture}
	\]
be a $\delta$--hyperbolic surface group extension.
If $H$ is a finitely generated infinite--index subgroup of $\pi_1(S)$, then $H$ is quasiisometrically embedded in $\Gamma_\Theta$.
\end{ScottSwarupGeneralized}
\begin{proof}
It suffices to show that a finite--index subgroup $H'$ of $H$ is quasiisometrically embedded.

By a theorem of Scott \cite{Scott.1978,Scott.1978.correction}, there is a finite cover $S'$ of $S$ in which the subgroup $H$ represents the fundamental group of a subsurface of $S'$.
Let $S''$ be the finite cover of $S$ such that 
	\[
		\pi_1(S'') = \bigcap_{\theta \in \Theta} \theta(\pi_1(S')).
	\]
Then $\Theta$ lifts to $\Mod(S'')$, and so there is a finite--index subgroup $\Gamma_\Theta'$ of $\Gamma_\Theta$ of the form
	\[
		\begin{tikzpicture}[>= to, line width = .075em, 
			baseline=(current bounding box.center)]
		\matrix (m) [matrix of math nodes, column sep=1.5em, row sep = 1em, 		text height=1.5ex, text depth=0.25ex]
		{
			1 & \pi_1(S'')  & \Gamma_\Theta' & \Theta & 1. \\
		};
		\path[->,font=\scriptsize]
		(m-1-1) edge (m-1-2)
		(m-1-2) edge (m-1-3)
		(m-1-3) edge (m-1-4)
		(m-1-4) edge (m-1-5)
		;
		\end{tikzpicture}
	\]
It follows immediately from Proposition \ref{P:qi subsurfaces general} that $H' = H \cap \pi_1(S'')$ is quasiisometrically embedded in $\Gamma_\Theta'$.
As the latter is finite--index in $\Gamma_\Theta$, it is quasiisometrically embedded there, and so $H'$ is quasiisometrically embedded in $\Gamma_\Theta$.
\end{proof}

\bibliographystyle{plain}
\bibliography{convex3manifolds}

\end{document}